\documentclass[12pt, reqno]{amsart}
\usepackage{amsmath}
\usepackage{amsfonts}
\usepackage{amssymb}
\usepackage[english]{babel}
\usepackage{color}

\usepackage{comment}
\usepackage{calrsfs}

\usepackage{graphicx}

\usepackage{verbatim}

\theoremstyle{plain}
\newtheorem {theorem}{Theorem}[section]
\newtheorem {lemma}[theorem]{Lemma}

\newtheorem {proposition} [theorem]{Proposition}
\theoremstyle{definition}
\newtheorem{definition}[theorem]{Definition}
\newtheorem{remark}[theorem]{Remark}

\theoremstyle{remark}

\newcommand{\R}{\mathbb R}

\renewcommand{\H}{H}

\newcommand{\Mu}{M}

\renewcommand{\epsilon}{\varepsilon}
\renewcommand{\phi}{\varphi}

\newcommand{\e}{\varepsilon}
\renewcommand{\theta}{\vartheta}
\newcommand{\la}{\lambda}

\newcommand{\C}{\mathbb C}

\newcommand{\barint}
{\rule[.036in]{.12in}{.009in}\kern-.16in \displaystyle\int}

\newcommand{\A}{\mathcal{A}}
\renewcommand{\L}{\mathcal {L}}

\usepackage{aurical}
\newcommand{\h} {\text{\large\Fontamici h}}

\long\def\MSC#1\EndMSC{\def\arg{#1}\ifx\arg\empty\relax\else
     {\par\narrower\noindent
     {\small\it 2010 Mathematics Subject Classification.} \small #1\par}\fi}

\long\def\KEY#1\EndKEY{\def\arg{#1}\ifx\arg\empty\relax\else
     {\par\narrower\noindent
     {\small\it Keywords and Phrases.} \small #1\par}\fi}

\newcommand{\X}{\mathcal X}

\newcommand{\Ric}{\mathrm{Ric}}
\newcommand{\q}{q}
\newcommand{\Hom}{\mathrm{Hom}}

\renewcommand{\la}{\langle}
\newcommand{\ra}{\rangle}
\newcommand{\enn}{\mathcal{N}}
\renewcommand{\div}{\mathrm{div}}

\renewcommand{\h}{\mathfrak{h}}

\newcommand{\M}{\mathcal M}


\begin{document}

\thanks{The authors thank the G.N.A.M.P.A. project: Variational Problems and
Geometric Measure Theory in Metric Spaces. }

\author[V.~Franceschi]{Valentina Franceschi}
\email{valentina.franceschi@inria.fr}

\author[F.~Montefalcone]{Francescopaolo Montefalcone}
\email{montefal@math.unipd.it}

\author[R.~Monti]{Roberto Monti}
\email{monti@math.unipd.it}

\address[Franceschi, Montefalcone, Monti]
{Universit\`a di Padova, Dipartimento di Matematica,
via Trieste 63, 35121 Padova, Italy}

\title[CMC spheres]{CMC spheres in the Heisenberg group}

\begin{abstract} We study a family of spheres with
constant mean curvature (CMC)  
in the Riemannian Heisenberg group $H^1$. These spheres
are conjectured to be the isoperimetric sets of $H^1$.
We prove several results supporting this conjecture.
We also focus our attention on the sub-Riemannian limit.
\end{abstract}

\maketitle

\tableofcontents

\section{Introduction}

\setcounter{equation}{0}

\renewcommand{\X}{{X}}
\newcommand{\T}{{T}}
\newcommand{\Y}{{Y}}

 In this paper, we study a family of spheres with
constant mean curvature (CMC)  
in the Riemannian Heisenberg group $H^1$. We introduce in $H^1$ two
real parameters
that can be used to deform $H^1$ to the sub-Riemannian Heisenberg group, on the
one
hand, and to the Euclidean space, on the other hand.
Even though we are not able to prove that these CMC spheres are in fact
isoperimetric sets, we obtain several partial results in this direction.
Our motivation comes from  the sub-Riemannian Heisenberg group, where it is conjectured that the solution
of the isoperimetric problem is obtained  rotating a Carnot-Carath\'eodory
geodesic around the center of the group, see \cite{Pan1}. This set is known as
Pansu's sphere. The conjecture is proved 
only assuming some regularity ($C^2$-regularity, convexity) or symmetry, see
\cite{CDST,FLM,M3,MR,R,RR}.

Given a real parameter $\tau\in\R$, let $\h = \mathrm{span}\{X,Y,T\}$
be the three-dimensional real Lie algebra
spanned by three elements $X,Y,T$ satisfying the relations $[X,Y] = -2\tau T$
and $[X,T] = [Y,T] = 0$. When $\tau\neq 0$, this is the Heisenberg Lie algebra
and we denote by $H^1$ the corresponding Lie group.
We will omit  reference to the parameter $\tau\neq0$ in our notation.
In suitable coordinates, we
can identify  $H^1$ with $\C\times\R$ and assume that $X,Y,T$ are left-invariant
vector fields in $H^1$ of the form
\begin{equation}\label{XYT}
X  = \frac 1 \epsilon\Big( \frac{\partial}{\partial x}+\sigma y
\frac{\partial}{\partial
t}\Big),
\quad
   Y  = \frac 1 \epsilon\Big( \frac{\partial}{\partial y}  - \sigma
x\frac{\partial}{\partial
t}\Big),
\quad\textrm{and}\quad 
     T  = \epsilon^2 \frac{\partial}{\partial t},
\end{equation}
where $(z,t)\in\C\times\R$ and $z=x+iy$. The  real parameters  
$\varepsilon>0$ and $\sigma\neq 0$
are  such that
\begin{equation}\label{tau-sigma}
 \tau \varepsilon^4=  \sigma .
\end{equation}
Let $\langle \cdot,\cdot\rangle$ be the scalar product on $\h$ making 
$X,Y,T$ orthonormal, that is extended to a left-invariant 
Riemannian metric $g=\langle \cdot,\cdot\rangle$
in $H^1$. 
The Riemannian volume of $H^1$ induced by this metric  
coincides with the Lebesgue measure $\mathcal L^3$ on $\C\times\R$ and, in fact,
it turns out to be independent  of $\varepsilon$ and $\sigma$ (and hence of
$\tau$).
    When
$\epsilon =1$ and $\sigma\to0$, the Riemannian manifold $(H^1,g)$ converges to
the Euclidean space. When $\sigma\neq 0$ and $\varepsilon\to0^+$, then $H^1$
endowed with the
distance function induced by the
rescaled metric $\varepsilon^{-2} \langle \cdot,\cdot\rangle$ converges   to the
sub-Riemannian Heisenberg group.  

The boundary of an isoperimetric region  is a surface with constant mean
curvature.
In this paper, we study  a family of  CMC spheres
$\Sigma_R\subset H^1$, with $R>0$, that foliate $H^1_* = H^1\setminus
\{0\}$, where $0$ is the neutral element of $H^1$. Each sphere $\Sigma_R$ is
centered
at $0$ and can be
described by an explicit formula that was first obtained by Tomter
\cite{T},
see Theorem \ref{TOM} below.
We  conjecture  that,
within its volume class and up to left
translations,
the sphere $\Sigma_R$ is  the unique solution of the isoperimetric problem
in $H^1$. When 
$\epsilon = 1$ and $\sigma\to0$, the spheres $\Sigma_R$ converge 
to the standard sphere of the Euclidean space. When $\sigma \neq 0$ 
is fixed and $\epsilon\to0^+$, the spheres $\Sigma_R$ converge to the Pansu's sphere.

In Section \ref{TRE}, we study some preliminary properties  of $\Sigma_R$,
 its second fundamental form and   principal  curvatures.
A central object  in this setting is the left-invariant $1$-form $\vartheta \in
\Gamma(T^* H^1)$
defined by 
\begin{equation} \label{1.1.1}
 \vartheta(V) = \langle   V,T\rangle\quad \text{for any $V\in\Gamma( TH^1)$.}
\end{equation} 
The kernel of $\vartheta$ is the horizontal distribution.
Let $N$ be the north pole of $\Sigma_R$ and $S=-N$ its south pole.
In $\Sigma_R^* = \Sigma_R\setminus \{\pm N\}$ there is an orthonormal frame of  
vector 
fields $X_1,X_2 \in \Gamma(T \Sigma_R^*)$ such that
$\vartheta(X_1)=0$, i.e., 
$X_1$ is a linear combination of $X$ and $Y$.
In Theorem \ref{3.1}, we  compute the second fundamental form of $\Sigma_R$ in
this frame. We show that the principal directions  of $\Sigma_R$ are given by 
a rotation of the frame $X_1,X_2$ by a \emph{constant} angle depending on the
mean curvature of $\Sigma_R$.

In Section \ref{S4},  we link in a continuous fashion the foliation
property of the Pansu's sphere with the foliation by meridians of the round
sphere in the Euclidean space. 
The foliation $H^1_* = \bigcup_{R>0} \Sigma_R$ determines a unit   vector field
$\enn\in \Gamma(TH^1_*)$ such that $\enn(p)\perp T_p\Sigma_R$ for
any $p\in\Sigma_R$ and $R>0$. The covariant derivative $\nabla\!_\enn\enn$,
where $\nabla$ denotes the Levi-Civita connection induced by the metric $g$,
measures how far the integral lines of $\enn$ are from being geodesics of $H^1$
(i.e., how far the CMC spheres $\Sigma_R$ are from being metric spheres). In
space forms, we would have $\nabla\!_\enn\enn=0$, identically. Instead, in
$H^1$ the normalized vector field  
\[
 \M(z,t) = \mathrm{sgn}(t) \frac{\nabla\!_\enn\enn}{|\nabla\!_\enn\enn|}, \qquad (z,t)  \in    \Sigma_R^*,
\]
is well-defined and smooth outside the center of $H^1$. In Theorem \ref{4.1},
we prove that for any $R>0$ we have
\[
 \nabla_\M^{\Sigma_R} \M = 0 \quad \textrm{on } \Sigma_R^*,
\]
where $\nabla^{\Sigma_R}$ denotes the restriction of $\nabla$ 
to $\Sigma_R$. This means that the integral lines of $\M$ are Riemannian 
geodesics of
$\Sigma_R$. In the coordinates associated with the frame \eqref{XYT}, when $\epsilon =1$ and $\tau=\sigma
\to 0 $  the integral lines of $\M$ converge   to the meridians
of the Euclidean sphere.
When $ \sigma\neq 0$ is fixed and $\epsilon\to0^+$, the vector field $\M$
properly normalized converges to the line flow of the geodesic foliation of the
Pansu's sphere, see Remark \ref{PAN}.

In Section \ref{k_0=0}, we give a proof of a known result that is announced in
\cite[Theorem 6]{A5} in the setting of three-dimensional homogeneous spaces (see
also \cite{MP}). Namely, we show
that 
any topological sphere
with constant mean curvature in $H^1$ is isometric to a CMC sphere $\Sigma_R$. 
The
proof
follows the scheme of the  fundamental paper 
\cite{AR}.
  
The surface $\Sigma_R$ is not totally umbilical and, for large enough $R>0$,  it
even
has negative
Gauss curvature near the equator, see Remark \ref{UMB}.
As a matter of fact, the distance from umbilicality is measured by a linear operator built up on the 
$1$-form  $\vartheta $. We can restrict  the tensor product
$\vartheta\otimes\vartheta$  to any surface $\Sigma$ in $H^1$ with constant
mean curvature $H$ and  then define, at any point $p\in \Sigma$, a symmetric linear
operator $k
\in \Hom(T_p\Sigma; T_p\Sigma)$ by setting
\begin{equation*} \label{kappa}
  k = h + \frac{2\tau^2}{\sqrt{H^2+\tau^2}} q_H\circ
(\vartheta\otimes\vartheta)\circ q_H^{-1},
\end{equation*}
where $h$ is the shape operator  of $\Sigma$ and 
$q_H$ is a rotation of each tangent plane $T_p\Sigma$ by  an angle that depends 
only on $H$, see formula \eqref{alpha_H}.  

In Theorem \ref{5.5}, we prove that for \emph{any} topological
sphere  $\Sigma\subset H^1 $ with constant mean curvature $H$,  the linear
operator $k$ on
$\Sigma$
satisfies the equation $k_0=0$. This follows from the Codazzi's
equations using Hopf's argument on holomorphic quadratic differentials, see
\cite{H}. The fact that $\Sigma$ is a left translation of $\Sigma_R$ now follows from
the analysis of the \emph{Gauss extension} of the topological sphere, see
Theorem \ref{5.9}.

 In some respect, it 
is an interesting issue to link the results of   Section  \ref{k_0=0}  with
the mass-transportation approach recently developed in \cite{BKS}.

In Section \ref{SEI}, we prove a stability result for the spheres $\Sigma_R$.
 Let $E_R\subset H^1 $ be the region bounded by $\Sigma_R$ and let $\Sigma\subset
H^1$  be the boundary of a smooth open set
$E\subset H^1$, $\Sigma = \partial E$, such that $\mathcal L^3(E) = \mathcal
L^{3}(E_R)$. Denoting
by $\A(\Sigma)$   the Riemannian area of $\Sigma$, we  conjecture  that
\begin{equation}
\label{isop}
  \A(\Sigma)- \A(\Sigma_R)\geq 0.
\end{equation}
We also conjecture that 
 a set $E$ is isoperimetric (i.e., 
    equality holds in \eqref{isop}) 
 if and only if  it is a left
translation of $E_R$.  We stress that if isoperimetric sets are topological
spheres, this statement would follow from Theorem
\ref{5.9}.

It is well known that isoperimetric sets are stable for perturbations fixing the
volume: in other words, the second variation
of the area is nonnegative. 
On the other hand,  using Jacobi fields arising from right-invariant vector fields of $H^1$,
it is possible to show that the spheres $\Sigma_R$ are stable with respect
to variations supported in suitable hemispheres, see Section \ref{SEI}.

In the case of the northern and southern hemispheres, we can prove a stronger 
form of
stability. Namely, using the coordinates associated with the frame \eqref{XYT}, for $R>0$ and
$0< \delta< R$ we  consider the cylinder
\[
 C_{\delta,R} = \big\{ (z,t) \in H^1 : |z|<R, t >f(R-\delta; R)\big\},
\]
where $f(\cdot;R)$ is the profile function of $\Sigma_R$, see \eqref{fuf}.
Assume that 
 the closure of 
$E\Delta E_R=E_R\setminus E  \cup E\setminus E_R$ is
a compact subset of $C_{\delta,R}$.
In Theorem \ref{thm:quant}, we prove that there exists a positive constant 
$C_{R\tau\epsilon}>0$ such that the following  quantitative
isoperimetric inequality holds:
\begin{equation} \label{quado}
 \A(\Sigma ) -\A(\Sigma_R) \geq \sqrt{\delta} C_{R\tau\epsilon} \mathcal
L^3(E\Delta E_R)^2.
 \end{equation}
The proof relies on a sub-calibration argument.
This provides further evidence on the conjecture that isoperimetric sets  
are precisely left translations of $\Sigma_R$.
When $\epsilon =1$ and $\sigma\to0$,  inequality 
\eqref{quado}
becomes a restricted form of the quantitative isoperimetric inequality in
\cite{FMP}. For fixed $\sigma\neq0$ and $\epsilon \to 0^+$ the rescaled area
$\epsilon\A$ converges to the sub-Riemannian Heisenberg perimeter and $\epsilon
C_{R\tau \epsilon}$ converges to a positive constant, see Remark \ref{6.2}. Thus inequality
\eqref{quado} reduces to the isoperimetric inequality proved in \cite{FLM}.

\section{Foliation of   $H^1_*$ by concentric stationary spheres}

\label{DUE}
\setcounter{equation}{0}

In this section, we compute the rotationally 
symmetric compact surfaces in $H^1$ that are area-stationary under a volume
constraint. We show that, for any $R>0$, there exists one such a sphere $\Sigma_R$ centered at
$0$. We will also show that 
$H^1_*=H^1 \setminus\{0\}$ is foliated by the  family of these
spheres, i.e., 
\begin{equation}\label{FOL}
    H^1 _* = \bigcup_{R>0} \Sigma_R.
\end{equation}
Each $\Sigma_R$ is given by an explicit formula that is due to Tomter, see
\cite{T}.

We work  in the   coordinates associated with the frame \eqref{XYT}, where the parameters
$\epsilon>0$ and $ \sigma \in\R$ are related by \eqref{tau-sigma}.
For any point $(z,t) \in  H^1$,  we set $r = |z| = \sqrt{x^2+y^2}$.

 \begin{theorem}\label{TOM}
 For any  $R>0$ there exists a unique  compact smooth embedded surface   $\Sigma_R\subset
\H^1$ that is area stationary under volume constraint  and such that
\[ 
\Sigma_R  = \{(z,t) \in
\H^1:  |t| =f(|z|;R) \}
\]
for a function $f(\cdot;R)\in C^\infty([0,R))$ continuous at $r=R$ with
$f(R)=0$. Namely, for any $0\leq r\leq R$ the function is given by
\begin{equation}
\label{fuf}
   f(r;R) = \frac{\e^2}{2\tau} \big[ \omega(R)^2  \arctan(p(r;R)) + \omega(r)^2
p(r;R)\big],
   \end{equation}
where  
\begin{equation*}
\label{omegap}
\omega(r) = \sqrt{1+\tau^2 \e^ 2 r^2}
   \quad
   \text{and}
   \quad
   p(r;R) = \tau\e
\frac{\sqrt{R^2-r^2}}{\omega(r)}.
\end{equation*}
 \end{theorem}

\begin{proof}
Let   $ D_R = \{z\in\C : |z|<R\}$  and for a  nonnegative radial function  
 $f \in  C^\infty(D_R)$  consider the graph  
$
   \Sigma = \{(z,f(z))\in\H^1: z\in D_R\}.
$
A frame of tangent vector fields   $V_1,V_2\in\Gamma(T\Sigma)$ is given by  
\begin{equation} \label{WIWA}
\begin{split}
V_1 = 
\e X  + \e^{-2}(f_x-\sigma y) T 
\quad
\text{and}
\quad
V_2  
 =   \e Y  +\e^{-2} (f_y+ \sigma x) T .
\end{split}
\end{equation}
Let $ g_\Sigma = \langle\cdot,\cdot\rangle$ be the restriction of the metric $g$
of $H^1$ to $\Sigma$.  Using the  entries of $g_\Sigma$ in the frame $V_1,V_2$,
we compute  the determinant
\begin{equation} \label{AER1}
 \begin{split}
 \det( g_\Sigma) &  
 = \e^{4}+ \e^{-2} \big\{ |\nabla f| ^ 2 + \sigma^2 |z|^2 + 2\sigma
(xf_y-yf_x)  \big\},
\end{split}
\end{equation}
where  $\nabla f=(f_x,f_y)$ is the
standard gradient 
of $f$ and $|\nabla f|$ is its length. We clearly have  $xf_y - yf_x  =0$ by the
radial symmetry of $f$.
Therefore, the area of $\Sigma$ is given by 
\begin{equation} \label{AER2}
A (f) = 
 \A (\Sigma) = \int_{D_R} \sqrt{\det( g_\Sigma)}\  dz = \frac1\e \int_{D_R}
\sqrt { \e^{6}+  
|\nabla f| ^ 2 + \sigma^2 |z|^2}\ dz ,
\end{equation}
where $dz$ is the Lebesgue measure in the $xy$-plane.

Thus, if $\Sigma$ is area stationary under a volume constraint, then
for any test function  $\varphi\in C^\infty_c(D_R)$ that is radially symmetric
and with vanishing mean (i.e.,  $\int _{D_R} \varphi\,
dz = 0$) we have 
\[
0 = \left.
 \frac{d}{ds} A(f+s\varphi)\right|_{s=0} = -\frac 1 \e \int_{D_R}  \varphi \,\mathrm{div}\Big(
\frac{\nabla f  }{\sqrt { \e^{6}+  
|\nabla f| ^ 2 + \sigma^2 |z|^2  }}\Big)
 \
 dz ,
\]
where
$\mathrm{div}$ denotes the standard divergence in the $xy$-plane. 
It follows that  there exists a constant   $H\in \R$ such that 
\begin{equation}\label{CURV}
  - \frac{1}{\e} \mathrm{div}\Big(
\frac{\nabla f  }{\sqrt { \e^{6}+  
|\nabla f| ^ 2 + \sigma^2 |z|^2  }}\Big) =  H.
\end{equation}
With abuse of notation we let $f(|z|) = f(z)$. Using   the radial variable $r=|z|$ and   the short notation   
\[
 g(r) = \frac{f_r} {r\sqrt{\e^6 + {f_r}^2 + \sigma^2 r^2}},
\]
the above equation reads as follows:
\[\frac{1}{r}\frac{d}{dr}
\big( r^2
g(r)\big)=\frac 1 r \big( r^2 g_r(r) + 2r g(r) \big) =
- \e H.
\]
Then there exists a constant $K\in \R$ such that  
$r^2 g = -\e r^2 H + K$.
Since  $g$ is bounded at   $r=0$, it must be  $K=0$ and thus  $g =- \e
 H$, and we get
\[
\frac{f_r} {r\sqrt{\e^6 + {f_ r}^2+ \sigma^2 r^2 }} = -\e H .
\]
From this equation, we see that $f_r$ has a sign.
Since   $\Sigma_R$  is compact, it follows that $H\neq 0$.
Since $f\geq 0$ we   have   $f_r<0$   and therefore $H>0$.

The surface  $\Sigma_R$  is smooth at the ``equator'' 
(i.e., where $|z|=R$ and $t=0$) and thus  we   have $f_r(R)
=-\infty$. As we will see later, this is implied by the relation 
\begin{equation}\label{eHR=1}
\epsilon HR=1, 
\end{equation}
that will be assumed throughout the paper.
Integrating the above equation 
we find
\begin{equation}\label{f_r}
 f_r(r) = - \e^4 Hr \sqrt{\frac{1+\tau^2\e^2 r^ 2}{1-\e^2 H^2 r^2}}
 = -\e^ 3 r \sqrt{\frac{1+\tau^2 \e^2 r^ 2}{R^2 - r^2}} ,\quad
 0\leq r<R.
\end{equation}
Integrating this expression on the interval $[r,R]$ and using   $f(R)=0$ we
finally find
\begin{equation}
\label{palix}
   f(r;R) = \e^ 3 \int_r^ R \sqrt{ \frac{1+\tau^2 \e^2 s^2}{R^2-s^2}} \, s ds.
\end{equation}
After some  computations, we obtain the explicit formula
\[
 f(r;R) = \frac{\e^2}{2\tau} \Big[  \omega(R)^2  \arctan\Big( \tau\e
\frac{\sqrt{R^2-r^2}}{\omega(r)}\Big) +\tau\e \omega (r) \sqrt{R^2-r^2} 
\Big],\quad
0\leq r\leq R,
\]
with   $\omega(r) = \sqrt{1+\tau^2 \e^2 r^2}$. This is formula \eqref{fuf}.
\end{proof}

\begin{remark} 
The function $f(\cdot;R) =f(\cdot;R;\tau;\epsilon)$ depends also on the
parameters $\tau$ and $\epsilon$, that are omitted in our notation.
With  $\epsilon =1$, we find  
\[
\lim_{\tau\to0} f(r;R;\tau;1) =\sqrt{R^2-r^2}.
\]
When $\tau \to0$, the spheres $\Sigma_R$ converge to Euclidean spheres with
radius $R>0$
in the
three-dimensional space.

With $\tau =\sigma/\epsilon ^4$ as in \eqref{tau-sigma}, we find the asymptotic
\[
\begin{split}
\lim_{\epsilon \to0} f(r;R;\sigma/\epsilon^4;\epsilon)  
&
=\frac{\sigma}{2}\Big[R^2\arctan\Big(\frac{\sqrt{R^2-r^2}}{r}\Big) +r
\sqrt{R^2-r^2}
\Big]
\\
&
=\frac{\sigma}{2}\Big[R^2\arccos\Big(\frac{r}{R}\Big) +r
\sqrt{R^2-r^2}
\Big] ,
\end{split}
\]
which  gives the profile function of  the  Pansu's sphere, the
conjectured solution 
to the sub-Riemannian Heisenberg isoperimetric problem, see e.g.~\cite{MR} or
\cite{M3}, with $R=1$ and $\sigma = 2$.
 \end{remark}

\begin{remark}
Starting from formula   \eqref{fuf}, we can compute the derivatives
of  $f(\cdot;R)$ in the variable  $R$. The first order derivative is given by 
\begin{equation}
\label{f_R}
 f_R(r;R) = \tau \epsilon^4 R   \Big[ \arctan\big( p(r;R) \big) +\frac{1}{
p(r;R)}\Big]= \frac{\sigma R}{p(r;R) \ell(p(r;R))},
\end{equation}
where $\ell:[0,\infty)\to\R$ is the function defined as
\begin{equation}
\label{ell}
\ell(p) = \frac{1}{1+ p\arctan(p)}.
\end{equation}
The geometric meaning of $\ell$ will be clear in formula \eqref{NR}.

\end{remark}

We now establish the foliation property \eqref{FOL}.

\begin{proposition}\label{fol}
For any nonzero   $(z,t) \in \H^1$ there exists a unique  $R>0$ such that 
$(z,t) \in \Sigma_R$. 
\end{proposition}

\begin{proof}
Without loss of generality we can assume that $t\geq 0$.
After an integration by parts in  \eqref{palix}, we obtain the formula  
\[
   f(r;R) = \e^ 3\Bigg\{ \sqrt{R^2-r^2} \omega(r) +\int_r^R \sqrt{R^2-s^2}
\omega_r(s) ds\Bigg\},\quad 0\leq r\leq R.
\]
Since $\omega_r(r)>0$ for $r>0$,
we deduce that the function   $R\mapsto
f(r;R)$ is strictly increasing for    $R\geq r$. Moreover, we have  
\[
 \lim_{R\to\infty} f(r;R) = \infty,
\]
and hence for any $r\geq 0$ there exists a unique  $R\geq r$ such that $f(r;R) =t$.
 \end{proof}

\begin{remark}
By Proposition \ref{fol}, we can define the function
 $R:\H^1 \to[0,\infty)$ by  letting $R(0)=0$ and $R(z,t) =R$ if and only if $(z,t)
\in \Sigma_R$ for   $R>0$. The function $R(z,t)$, in fact, depends on
$r=|z|$ and thus we may consider $R(z,t)=R(r,t)$ as a function of $r$ and $t$.
This function is implicitly defined by the equation $ 
       |t| = f(r; R(r,t))$. Differentiating this equation, we find the derivatives
of $R$, i.e.,
       \begin{equation}\label{R_t}
R_r  = -\frac{f_r}{f_R}
\quad \text{and}
\quad
R_t = \frac{\mathrm{sgn}(t)}{f_R}, 
 \end{equation}
where $f_R$ is given by \eqref{f_R}.
\end{remark}

\section{Second fundamental form of $\Sigma_R$}
\label{TRE}
 \setcounter{equation}{0}

In this section, we  compute the second fundamental form of the   spheres 
$\Sigma_R$. In fact, we will see that  $H =1/(\varepsilon R)$ is the mean
curvature of $\Sigma_R$, as already clear from \eqref{CURV} and \eqref{eHR=1}.  
Let $N= (0, f(0;R)) \in\Sigma_R$ be the north pole of $\Sigma_R$ and 
let $S = -N = (0,- f(0;R))$ be its   south pole.
In $\Sigma_R^*=\Sigma_R\setminus \{\pm N\}$ there is a   frame of  tangent 
vector 
fields $X_1,X_2 \in \Gamma( T \Sigma_R^*)$ such that
\begin{equation}
\label{pipo}
|X_1| = |X_2| = 1,\quad \langle X_1,X_2\rangle = 0,\quad \vartheta(X_1)=0,
\end{equation}
where $\vartheta$ is the left-invariant $1$-form introduced in \eqref{1.1.1}.
Explicit expressions for $X_1$ and $X_2$ are given in formula \eqref{xixo}
below.
This frame is unique up  to the sign  $\pm X_1$ and $\pm X_2$.
Here and in the rest of the paper, we denote by $\enn$ the exterior  unit normal
to the spheres $\Sigma_R$.

The second fundamental form $h$ of  $\Sigma_R$ with respect to the frame
$X_1,X_2$  is given by 
\[
  h=(h_{ij})_{i,j=1,2},\qquad  h_{ij}  = \langle \nabla_{X_i} \mathcal N,X_j\rangle,\quad i,j=1,2,
\]
where $\nabla $ denotes the Levi-Civita connection of $H^1$
 endowed with the left-invariant metric $g$.
The linear connection $\nabla$ is represented 
by the  linear mapping   $\h\times\h\mapsto \h$, $(V,W) \mapsto \nabla _{V} W$.
Using the fact that the connection is torsion free and metric, it can be  
seen that $\nabla$ is characterized by the following relations:
\begin{equation}\label{FR}
\begin{split}
 &\nabla_XX =\nabla_Y Y=\nabla_TT=0,
\\
&
 \nabla _Y X = \tau  T 
 \quad 
 \textrm{and}\quad
  \nabla _X Y =- \tau T,
\\
&
\nabla _T X = \nabla _XT = \tau Y,
\\
&
\nabla _T Y = \nabla _YT = -\tau X.
\end{split}
 \end{equation}
Here and in the rest of the paper, 
we use  the coordinates associated with the frame \eqref{XYT}.
For $(z,t) \in H^1$, we set $r = |z|$
and use the short notation
\begin{equation} \label{RHO}
  \varrho = \tau \epsilon r.
\end{equation}

\begin{theorem} \label{3.1}
For any $R>0$, the second fundamental form $h$ 
of $\Sigma_R$ with respect to the frame   $X_1,X_2$ in \eqref{pipo} at the point
  $(z,t) \in
\Sigma_R$ 
is given by
\begin{equation} \label{ACCA}
h = \frac{1}{1+\varrho^2}
\left(
\begin{array}{cc}
H ( 1+2\varrho^2) & \tau \varrho^2
\\
\tau \varrho^2 & H
\end{array}
\right),
\end{equation}
where   $R=1/H\e$ and $H$ is the mean curvature of $\Sigma_R$.
The principal curvatures of $\Sigma_R$ are given by
\begin{equation}\label{kappa_12}
\begin{split}
 \kappa_1 & = H + \frac{\varrho^2}{1+\varrho^2} \sqrt{H^2 + \tau^2}  ,
\\
\kappa_2 & = H - \frac{\varrho^2}{1+\varrho^2} \sqrt{H^2 + \tau^2}.
\end{split}
\end{equation}
Outside the north and south poles,    principal directions  are given by
\begin{equation}
\label{K_12}
\begin{split}
   K_1 & = \cos\beta X_1+\sin\beta X_2,
   \\
   K_2 & = -\sin\beta X_1+\cos\beta X_2,
   \end{split}   
\end{equation}
where $\beta  =\beta_H\in (-\pi/4,\pi/4)$ is the angle
\begin{equation}\label{beta_H}
\beta_H =\arctan \Bigg(\frac{\tau   }{H+\sqrt{ H^2+\tau^2 }}\Bigg).
\end{equation}
\end{theorem}

\begin{proof} 

Let  $a,b:\Sigma_R^*\to\R$ and $c,p:\Sigma_R\to\R$ be the
following 
functions depending on the radial coordinate $r=|z|$:  
\begin{equation}
\label{abcp}
\begin{split}
a & =a(r;R) = \frac{\omega(r)}{r\omega(R)},
\quad
b =b(r;R)  = \pm \frac{\sqrt{R^2-r^2} } {r R \omega(R)},
\\
c &=c(r;R)  =\frac{r\omega(R)  } {  R \omega(r)},
\quad
p=p(r;R) =  \pm \tau \epsilon \frac{\sqrt{R^2-r^2} } { \omega(r)}.
\end{split}
\end{equation}
In fact,
$b$ and $p$ also depend on the sign of $t$. Namely, in $b$ and   $p$ we choose
the sign $+$  in the northern hemisphere, that is for $t\geq0$, while
we choose  the sign $-$  in the southern hemisphere, where  $t\leq0$.  
Our computations are in the case $t\geq0$.

The vector fields
\begin{equation}
\label{xixo}
\begin{split}
X_1 &  = - a \big(( y-xp) X - (x+yp) Y\big),
\\
X_2   & =- b \big(  (x+yp) X+( y-xp) Y\big) + c T
\end{split}
\end{equation}
form an orthonormal frame for $T\Sigma_R^*$ satisfying
\eqref{pipo}. This frame can be computed starting from \eqref{WIWA}.
The outer unit normal to $\Sigma_R$ is given by  
\begin{equation}
\label{ENNE}
\begin{split}
\mathcal N
 =\frac{1}{R} \Big\{ (x+y p) X +(y-xp) Y +\frac{p}{\tau\epsilon}T\Big\}.
\end{split}
\end{equation}
Notice that this formula is well defined also at the poles.

We compute the entries $h_{11} $ and $h_{12}$. Using $X_1 R=0$, we find 
\begin{equation}
\label{M1}
\begin{split}
  \nabla_{X_1} \mathcal N= \frac 1 R \Big\{
  & X_1(x+yp) X + X_1(y-xp) Y + X_1\Big(\frac{p}{\tau\epsilon}\Big) T
\\
& 
+(x+yp) \nabla_{X_1} X + (y-xp) \nabla_{X_1} Y +\frac{p}{\tau\epsilon}
\nabla_{X_1} T \Big\},
\end{split}
\end{equation}
where, by the fundamental relations \eqref{FR},
\begin{equation}\label{M2}
\begin{split}
&\nabla_{X_1} X =\tau  a(x+yp) T,\quad
\\
&
\nabla_{X_1} Y =\tau  a(y-xp) T,
\\
&
\nabla_{X_1} T =-\tau a\big[  ( y-xp) Y+(x+yp) X\big].
\end{split}
\end{equation}
Using the formulas
\[
X_1 x = -\frac{a}{\epsilon} ( y-xp )\quad\text{and}
\quad
X_1 y =\frac{a}{\epsilon} (x+yp), 
\]
we find the derivatives 
\begin{equation}\label{M3}
\begin{split}
X_1(x+yp) & = \frac{a}{\epsilon}\big(2xp +y(p^2-1)\big)+yX_1 p,
\\
X_1(y-xp ) & = \frac{a}{\epsilon}\big(2yp +x(1-p^2 )\big)-x X_1 p.
\end{split}
\end{equation}
Inserting \eqref{M3} and \eqref{M2} into \eqref{M1}, we obtain
\begin{equation}
\label{M4}
\begin{split}
  \nabla_{X_1} \mathcal N= \frac 1 R \Big\{
   \Big[ -\frac{a}{\epsilon} (y-xp) +y X_1p\Big] X
 & + \Big[ \frac{a}{\epsilon} (x+yp) -x  X_1p\Big]Y
\\
&
  + \Big[ \frac{X_1 p}{\tau\epsilon} +\tau r^ 2 a (p^2+1)\Big] T
  \Big\}.
\end{split}
 \end{equation}
 From this formula we get
 \[
 h_{11} =\langle  \nabla_{X_1} \mathcal N ,X_1\rangle
  =\frac{r^2 a}{R\epsilon} \big\{ a (p^2+1)-\epsilon X_1p\big\},
 \]
 where  
   $p^2+1 = \omega(R)^2/\omega(r)^2$ and   $X_1p$ can be computed starting from 
 \begin{equation}
 \label{p_r}
    p_r(r;R) = -\tau\epsilon r\frac{\omega(R)^2}{\sqrt{R^2-r^2} \omega(r)^3}.
    \end{equation}
    Namely, also using the formula for $a$ and $p$ in \eqref{abcp}, we have
    \[
    X_1 p =\frac{r a}{\epsilon} p p_r   
  =  
 -\tau^2\epsilon r \frac{\omega(R)}{   \omega(r)^3}.
 \]
With \eqref{eHR=1} and \eqref{RHO}, we  finally  find
\[
 h_{11} =\frac{1}{R\epsilon}\Big( 1+ \frac{\tau^2\epsilon^2
r^2}{\omega(r)^2}\Big) =  H \Big( 1+\frac{\varrho^2}{1+\varrho^2}\Big).
\]

>From \eqref{M4} we also deduce
\[
h_{12}=\langle \nabla_{X_1} \mathcal N ,X_2\rangle
=-\frac{b}{R}r^2pX_1p+\frac{c}{R}\Big\{\frac{X_1p}{\tau\varepsilon}+\tau
r^2a(1+p^2)\Big\},
\]
and using the formula for $X_1p$ and the formulas in \eqref{abcp} we
obtain
\[
h_{12} 
=\frac{\tau\varrho^2}{1+\varrho^2}
.
\]

To compute the entry $h_{22}$, we start from 
\begin{equation}
\label{eq:N1}
\begin{split}
\nabla_{X_2}\mathcal N
=\frac{1}{R}\Big\{&
X_2(x+yp)X+X_2(y-xp)Y+\frac{X_2(p)}{\tau \varepsilon}T\\
&\quad
+(x+yp)\nabla_{X_2}X+(y-xp)\nabla_{X_2}Y
+\frac{p}{\tau\varepsilon}\nabla_{X_2}T
\Big\},
\end{split}
\end{equation}
where, by \eqref{FR} we have
\begin{equation}
\label{eq:N2}
\begin{split}
\nabla_{X_2}X&=-\tau b(y-xp )T+\tau c Y,\\
\nabla_{X_2}Y&=\tau b(x+yp)T-\tau c X,\\
\nabla_{X_2}T&=-\tau b(x+yp)Y+\tau b(y-xp ) X.
\end{split}
\end{equation}
Since $X_2x=-b(x+yp )/\varepsilon$ and $X_2y=-b(y-xp )/\varepsilon$, we
get
\begin{equation}
\label{eq:N3}\begin{split}
X_2(x+yp)&=-\frac{b}{\varepsilon}\big(2yp+x(1-p^2)\big)-yX_2p,\\
X_2(y-xp)&=\frac{b}{\varepsilon}\big(2xp+y(p^2-1)\big)+xX_2p.
\end{split}
\end{equation}
Inserting \eqref{eq:N2} and \eqref{eq:N3} into \eqref{eq:N1} we obtain
\begin{equation*}
\label{eq:N4}
\begin{split}
\nabla_{X_2}\mathcal N
=\frac{1}{R}\Big\{
-&\Big[\frac{b}{\varepsilon}(x+yp)+yX_2p+\tau c(y-xp)\Big]X\\
  +&\Big[-\frac{b}{\varepsilon}( y-xp )+xX_2p+\tau c (x+yp)\Big]Y
-\frac{X_2p}{\tau\varepsilon}T
\Big\},
\end{split}
\end{equation*}
and thus 
\[
h_{22}=\langle \nabla_{X_2}\mathcal N, X_2\rangle
=
\frac{br^2}{\varepsilon R}\big\{b(1+p^2)+\varepsilon pX_2p\big\}-\frac{c
X_2p}{\tau\varepsilon R}.
\]
Now   $X_2p$ can be computed by using  \eqref{p_r} and   the formulas
\eqref{abcp}, and we obtain
\[
X_2p=-\frac{\tau r\omega(R)}{R\omega(r)^3}.
\]By \eqref{eHR=1} and \eqref{RHO} 
we then conclude that
\[
h_{22} 
=\frac{H}{1+\varrho^2}.
\]

The principal curvatures $\kappa_1,\kappa_2$ of $\Sigma_R$
are the solutions to the system
\[
\begin{split}
\left\{
\begin{array}{l}
\kappa_1+\kappa_2 =\mathrm{tr}(h) =2H
\\
\displaystyle \kappa_1\kappa_2 =\det(h) = \frac{H^2(1+2\varrho^2)
-\tau^2\varrho^4}{(1+\varrho^2)^2}.
\end{array}
\right.
\end{split}
\]
They are given explicitly by the formulas \eqref{kappa_12}.

Now let $K_1,K_2$ be tangent vectors as in \eqref{K_12}.
We identify  $h$ with the shape operator   $h\in \Hom( T_p\Sigma_R;
T_p\Sigma_R)$, $h(K) = \nabla_K\enn$, at any point $p\in \Sigma_R$ and   $K\in
T_p\Sigma_R$. 
When
$\varrho\neq0$ (i.e., outside the north and south poles), the system of
equations
\[
h(K_1) =\kappa_1 K_1\quad \text{and}\quad
h(K_2) =\kappa_2 K_2
\]
is satisfied if and only if the angle     $\beta=\beta_H $ is chosen as in
\eqref{beta_H}. The argument of $\arctan$ in \eqref{beta_H}
 is in the interval $(-1,1)$ and thus $\beta_H \in (-\pi/4,\pi/4)$.
\end{proof}

\begin{remark}
\label{UMB} When $2H^2 <(\sqrt{5}-1)\tau^2$,
the  set of points $(z,t) \in \Sigma_R$
such that 
\[
  \varrho^2   > \frac{H}{\sqrt{H^2+\tau^2}-H}
\]
is nonempty. The inequality above  is equivalent to $\kappa_2<0$ at the point
$(z,t) \in\Sigma_R$. This means that,  for large enough $R$, points in
$\Sigma_R$ near the equator 
have strictly negative Gauss curvature.
\end{remark}

\begin{remark}
The convergence of the Riemannian second fundamental form towards its 
sub-Riemannian counterpart is studied in \cite{CPT}, in the setting of Carnot
groups.
\end{remark}

\section{Geodesic foliation of $\Sigma_R$}
\label{S4}
\setcounter{equation}{0}

We prove that each CMC sphere $\Sigma_R$ is foliated by a family of geodesics
of $\Sigma_R$ joining the north to the south pole.
In fact, we show that the foliation is governed by the normal $\enn$ to the
foliation $H^1_*=\bigcup_{R>0}\Sigma_R$. In the sub-Riemannian limit, we
recover the foliation property of the Pansu's sphere. 
In the Euclidean limit, we find the foliation of the round sphere with
meridians.

We need two preliminary lemmas.
We define a 
function $R: H^1\to [0,\infty)$ by
letting $R(0)=0$ and $R(z,t) = R$ if and only if $(z,t)\in\Sigma_ R$.
In fact, $R(z,t)$ depends on    $r=|z|$ and $t$.
The function $p$ in \eqref{abcp} is of the form $p=p(r,R(r,t))$.

  Now, we compute the
derivative of these functions in the normal direction $\mathcal N$.

\begin{lemma}
The derivative along $\mathcal N$ of the functions $R$ and $p$ are,
respectively,
\begin{equation}\label{NR}
 \mathcal N R  
 =\frac{\ell(p)}{\epsilon},
\end{equation}
and 
\begin{equation}\label{Np}
  \mathcal N p = \epsilon \tau^2 
  \frac{ R^2 \omega(r)^2 \ell(p)-r^2 \omega(R)^2 }{R\omega(r)^4p} ,
\end{equation}
where $\ell(p) = (1+p\arctan p) ^{-1} $, as    in \eqref{ell}.
\end{lemma}

\begin{proof} We start from  
the following expression for the unit normal 
(in the coordinates $(x,y,t)$):
\begin{equation*}
\label{ENNEradiale}
\mathcal N = \frac 1 R \Big\{ \frac r \epsilon  \partial _r +\frac p\epsilon
(y\partial_ x-x\partial_y) +\mathrm{sgn}(t) \epsilon ^2 \omega(r) \sqrt{R^2-r^2}
\partial_t\Big\}.
\end{equation*}
We just consider the case $t\geq 0$. Using \eqref{R_t}, we obtain
\[
\begin{split}
\mathcal N R & = \frac 1 R \Big\{
\frac r \epsilon R_r +\epsilon ^ 2 \omega(r) \sqrt{R^2-r^2} R_t\Big\}
=\frac {1 }{R f_R} \Big\{\epsilon ^ 2 \omega(r) \sqrt{R^2-r^2}  
- \frac r \epsilon  {f_r}  \Big\}.
\end{split}
\]
Inserting into this formula the expression in  
 \eqref{f_r} for $f_r$ we get
\[
\mathcal N R = \frac{\e^2 R \omega(r)}{f_R \sqrt{R^2-r^2}},
\]
and using formula \eqref{f_R} for $f_R$, namely,  
\[
 f_R =\tau \epsilon ^ 4 R\Big[ \arctan (p)+\frac 1 p \Big]
 =\frac{\tau \epsilon ^ 4   R}{p \ell(p)},
\]
we   obtain formula \eqref{NR}.

To compute the derivatives of $p$ in $r$ and $t$, we have to consider $p=p(r;R)$ and $R
= R(r,t)$. Using  
the formula in \eqref{abcp} for $p$ and the expression
\eqref{R_t} for $R_r$ yields  
\[ 
 p_r =-\frac{\tau\epsilon r \omega(R)^2}{\omega(r)^3 \sqrt{R^2-r^2}},\quad
 p_R =\frac{\tau\epsilon R}{\omega(r)\sqrt{R^2-r^2}},\quad
 R_r = -\frac{f_r}{f_R} = \frac{\epsilon^3 r\omega(r)} {\sqrt{R^2-r^2} f_R},
\]
   and thus
\[
\begin{split}
\frac{\partial}{\partial r} p(r,R(r,t)) & 
= p_r (r,R(r,t))+p_R (r,R(r,t))R_r(r,t)
\\&
= 
\frac{\tau\epsilon r }{\omega(r) ^ 3 \sqrt{R^2-r^2} }\big[
\omega(r)^2 \ell(p) -\omega(R)^2\big].
\end{split}
\]
Similarly, we compute
\[
\frac{\partial}{\partial t}  p(r;R(r,t)) = p_R(r;R(r,t))  R_t(r,t)= \frac{\tau
\ell(p)}{\epsilon ^2\omega(r)^2}.
\]
The derivative of $p$ along $\mathcal N$ is thus
as in \eqref{Np}, when $t\geq 0$. The case $t<0$ is analogous. 

\end{proof}

In the next lemma, we compute the covariant derivative
$\nabla\!\! _{\mathcal N}\mathcal N $. The resulting vector field
in $H^1_*$ is tangent to each CMC sphere $\Sigma_R$, for any $R>0$.

\begin{lemma} At any point in $(z,t)\in H^1_*$ we have 
\begin{equation}
\label{nablaENNE}
\nabla\!\! _{\mathcal N}\mathcal N (z,t)= \mathcal N\Big(\frac p R\Big)
 \Big[ (y+x\Phi)X-(x-y\Phi ) Y+\frac{1}{\tau\epsilon} T\Big],
\end{equation}
where $\Phi=\Phi(r;R)$ is the function defined as
\[
\Phi = - \frac{\omega(r) ^2 p }{\tau^2\epsilon^2 r^2},
\]
and the derivative $\mathcal N(p/R)$ is given by 
\[
 \mathcal N\Big(\frac p R\Big) = - \frac{\epsilon \tau^2 r^2 \big( \omega(R)^2 -\ell(p)
\omega(r)^2\big)}{R^2 \omega(r)^4 p},
\]
with $\ell$ as in \eqref{ell}.
\end{lemma}

\begin{proof}
Starting from formula \eqref{ENNE} for $\mathcal N$, we find that
\begin{equation}\label{G0}
\begin{split}
\nabla\!\! _{\mathcal N}\mathcal N &
=\mathcal N\Big(\frac{x+yp}{R}\Big) X+\mathcal N\Big(\frac{y-xp }{R}\Big)Y
+\mathcal N\Big(\frac{p}{\tau \epsilon R}\Big)T
\\
&+\frac 1 R \Big((x+yp)
\nabla\!\! _{\mathcal N} X+
(y-xp)
\nabla\!\! _{\mathcal N} Y+
\frac{p}{\tau\epsilon}
\nabla\!\! _{\mathcal N} T\Big),
\end{split}
\end{equation}
where, by the fundamental relations \eqref{FR}, we have
\begin{equation}\label{G1}
(x+yp)
\nabla\!\! _{\mathcal N} X+
(y-xp)
\nabla\!\! _{\mathcal N} Y+
\frac{p}{\tau\epsilon}
\nabla\!\! _{\mathcal N} T=\frac{2p}{\epsilon R} \Big(-(y-xp)X +(x+yp)Y\Big).
\end{equation}

>From the elementary formulas
\[
  \mathcal N x = \frac{1}{R\epsilon} (x+yp)
  \quad \text{and} \quad
  \mathcal N y =\frac{1}{R\epsilon} (y-xp),
\]
we find 
\begin{equation}\label{G2}
\begin{split}
\mathcal N (x+yp) &  = \frac{1}{\epsilon R} 
\big(x(1-p^2) +2yp \big) +y\mathcal N p,
\\
\mathcal N (y-xp) & = \frac{1}{\epsilon R} 
\big(y(1-p^2) -2xp \big) -x\mathcal N p.
\end{split}
\end{equation}
Inserting \eqref{G1} and \eqref{G2} into \eqref{G0} we obtain
the following expression  
\begin{equation}\label{G3}
\begin{split}
\nabla\!\! _{\mathcal N}\mathcal N 
=\frac{1}{R^2} \Big[
&\Big\{ x ( \epsilon^{-1} (1+p^2) -\mathcal NR)  
+ y (R\mathcal N p -p\mathcal N R)  \Big\} X
\\
+&
\Big\{ y ( \epsilon^{-1} (1+p^2) -\mathcal NR)  
-x (R\mathcal N p -p\mathcal N R)  \Big\} Y
\\
+&\frac{1}{\tau\epsilon} 
(R\mathcal N p -p\mathcal N R)  T\Big]  .
 \end{split}
 \end{equation}

>From \eqref{NR} and \eqref{Np} we compute
\[
R\mathcal N p -p\mathcal N R = - \frac{\epsilon \tau^2 r^ 2}{\omega(r)^4
p}\big[ \omega(R)^2 -\ell(p) \omega(r)^2\big].
\]
Inserting this formula into \eqref{G3} and using $1+p^2
=\omega(R)^2/\omega(r)^2$ yields the claim.
\end{proof}

Let $\mathcal N\in \Gamma(T H^1 _*)$ be the exterior unit normal to the family
of CMC spheres
$\Sigma_R$ centered at $0\in H^ 1$.  
The vector field $\nabla \! _{\mathcal N}\mathcal N$ is  tangent to $\Sigma_R$ for any $R>0$, and
for $(z,t)\in\Sigma_R$  we have 
\begin{center}
$ \nabla\! _{\mathcal N}\mathcal N(z,t) =0\quad$ if and 
only if $\quad z=0\,$ or $\, t=0$. 
\end{center}
 However, it can be checked that the normalized vector field
\[
\mathcal M (z,t) = \mathrm{sgn}(t) \frac{\nabla\! _{\mathcal N}\mathcal N }
{|\nabla\! _{\mathcal
N}\mathcal N |}\in \Gamma( T\Sigma_R^*)
 \]
is  smoothly  defined also at points $(z,t)\in \Sigma_R$ at the equator, 
where $t=0$.
We denote by $\nabla^{\Sigma_R}$ the restriction of  the Levi-Civita connection
$\nabla$ to $\Sigma_R$.

\begin{theorem} \label{4.1}  Let $\Sigma _R \subset H^1$ be the CMC sphere with
mean curvature
$H>0$.  Then the vector field $\nabla\! _{\mathcal M}\mathcal M $ is smoothly
defined on   $\Sigma_R$ and   for any $(z,t)\in \Sigma_R$ we have
\begin{equation} \label{espo}
\nabla\!_{\mathcal M}\mathcal M (z,t) = - \frac{H}{ \omega(r)^ 2}\mathcal N.
\end{equation}
In particular,   $\nabla^{\Sigma_R} _{\mathcal M}\mathcal M =0$ and the
integral curves of $\mathcal M$ are Riemannian geodesics of $\Sigma_R$ joining
the north pole
$N$ to the south pole $S$.

\end{theorem}

\begin{proof}
From \eqref{nablaENNE}   we obtain the following formula for $\mathcal M$:
\begin{equation}\label{C0}
\mathcal M = (x\lambda-y\mu) X + (y\lambda+x \mu) Y -\frac{\mu}{\tau\epsilon} T,
\end{equation}
where $\lambda,\mu:\Sigma_R^* \to \R$ are the functions
\begin{equation}
\label{C}
\lambda =
\lambda(r)=\pm \frac{\sqrt{R^2-r^2}}{rR}\quad \text{and}\quad
\mu = \mu(r)=\frac{\tau\epsilon r }{R \omega(r)}, 
\end{equation}  
with $r=|z|$ and $R=1/(\epsilon H)$.   
The functions $\lambda$ and $\mu$ are radially symmetric in $z$.
In defining $\lambda$ we choose the sign $+$, when $t\geq 0$, and the sign $-$,
when $t<0$.
In the coordinates $(x,y,t)$, the vector field $\mathcal M$ has the following
expression  
\begin{equation}\label{star}
\mathcal M =\frac 1 \epsilon \Big( \lambda r \partial _r+\mu
(x\partial_y-y\partial_x) -
\mu \frac{\epsilon^2\omega(r)^2}{\tau} \partial _t\Big),
\end {equation}
where $r\partial _r =x\partial_x+y\partial_y$, and so we have
\begin{equation}\label{phil}
  \begin{split}
\nabla\!_{\mathcal M}\mathcal M 
= &  (x\lambda-y\mu) \nabla\!_{\mathcal M} X 
+
(y\lambda+x\mu) \nabla\!_{\mathcal M} Y-\frac{\mu}{\tau\epsilon}
\nabla\!_{\mathcal M} T
\\
&
+ \mathcal M (x\lambda-y\mu) X
+ \mathcal M (y\lambda+x\mu) Y
- \mathcal M \Big( \frac{\mu}{\tau\epsilon} \Big) T.
  \end{split}
\end{equation}
  Using \eqref{star}, we compute
\begin{equation}\label{1}
  \mathcal M x = \frac1\epsilon (x\lambda-y\mu) \quad \text{and}\quad
  \mathcal M y = \frac1\epsilon (y\lambda+x\mu),
\end{equation}
and so we find
\begin{equation}\label{2}
\begin{split}
\mathcal M (x\lambda-y\mu) &= \frac 1 \epsilon (x\lambda-y\mu)\lambda  +
x\mathcal M \lambda -\frac1\epsilon (y\lambda+x\mu)\mu  -y\mathcal M \mu,
\\
\mathcal M (y\lambda+x\mu) &= \frac 1 \epsilon (y\lambda+x\mu)\lambda  +
y\mathcal M \lambda +\frac1\epsilon (x\lambda-y\mu)\mu  +x\mathcal M \mu.
\end{split}
\end{equation}
Now, inserting \eqref{1} and \eqref{2} into \eqref{phil}, we get
\begin{equation*}
\label{3}
\begin{split}
\nabla\!_{\mathcal M}\mathcal M = & \Big( \frac x \epsilon
(\lambda^2+\mu^2)+x\mathcal M \lambda-y\mathcal M\mu\Big) X 
\\
  +& 
\Big( \frac y \epsilon (\lambda^2+\mu^2) +y\mathcal M \lambda +x\mathcal
M\mu\Big)Y
-\frac{1}{\tau\epsilon} \mathcal M \mu T.
\end{split}
\end{equation*}

The next computations are for the case $t\geq 0$.
Again from \eqref{star}, we get
\begin{equation}\label{B}
  \mathcal M \lambda  = \frac{\lambda r}{\epsilon} \partial_r\lambda = -
\frac{R\lambda}{\epsilon r\sqrt{R^2-r^2}},
\quad \text{and}\quad
  \mathcal M \mu = \frac{\lambda r}{\epsilon}
 \partial _r \mu=
\frac{\tau r \lambda}{R \omega(r)^3}.
\end{equation}
From \eqref{C} and \eqref{B} we have 
\[
\frac{1}{\epsilon} (\lambda^2+\mu^2)+\mathcal M\lambda =
-\frac{1}{\epsilon R^2 \omega(r)^2},
\] and so we finally obtain
\begin{equation}
\label{NUM1}
\nabla\!_{\mathcal M}\mathcal M =(x\Lambda-y\Mu) X +(y\Lambda +x\Mu) Y
-\frac{\Mu}{\tau\epsilon} T,
\end{equation}
where we have set 
\begin{equation}
\label{NUM2}
  \Lambda = -\frac{1}{\epsilon R^2 \omega(r)^2},\qquad
  \Mu =\tau
\frac{\sqrt{R^2-r^2}}{R^2\omega(r)^3}.
\end{equation}
Comparing with \eqref{ENNE}, we deduce that 
\[
\nabla\!_{\mathcal M}\mathcal M = - \frac{1}{\epsilon R\omega(r)^ 2}\mathcal N.
\]
The claim $\nabla^{\Sigma_R} _{\mathcal M}\mathcal M =0$ easily
 follows from the last formula.

\end{proof}

\begin{figure}[h!] 
\includegraphics[scale=0.58]{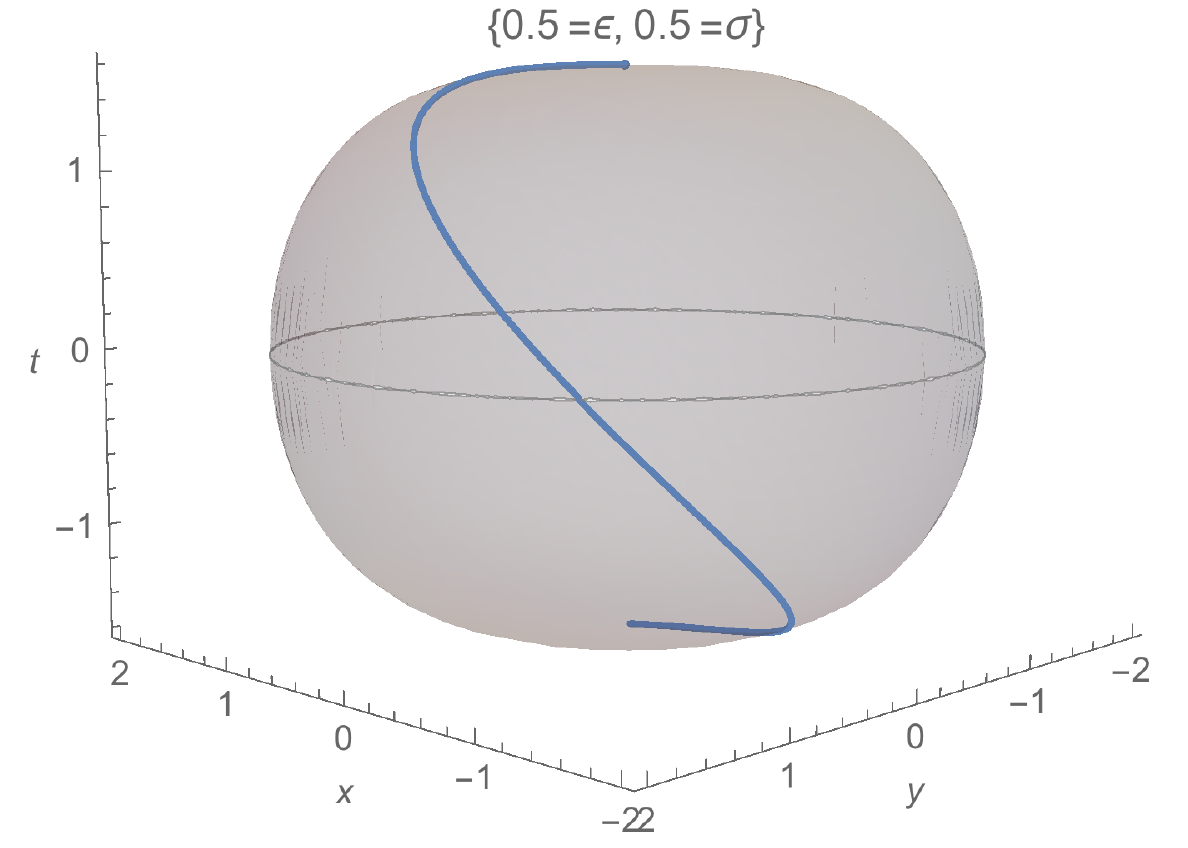} 
\caption{The plotted curve is  an integral curve of the vector field
$\mathcal M$ for
$R=2$,  $\varepsilon=0.5$, and $\sigma=0.5$.}
\end{figure}

\begin{remark} We compute the pointwise limit of $\mathcal M$ 
 in \eqref{C0} when $\sigma\to0$, for $t\geq 0$. In the southern hemisphere
the situation is analogous. By \eqref{star}, the vector field $\mathcal M$  is
 given by    
\[
\begin{split}
 \mathcal M =  \frac {1}{ \epsilon R}  \Big( 
\frac{\sqrt{R^2-r^2}}{r} (x \partial_ x  +y  \partial_y )
+  {\frac{\sigma r}{\sqrt{\epsilon^6+
\sigma^2 r ^2}}}  ( x \partial_y  -y \partial_x ) 
 - r \sqrt{\epsilon^6+ \sigma ^2r ^2}  {\partial_t}\Big) .
\end{split}
\]
With  $\epsilon=1$ we have  
\[
 \widehat{\mathcal M}
   =\lim_{\sigma  \to 0 } \mathcal M= 
\frac{\sqrt{R^2-r^2}}{rR} (x \partial_ x  +y  \partial_y ) 
-\frac r R {\partial_t}.
\]
Clearly, the vector field $\widehat {\mathcal M}$ is tangent to the round sphere of
radius
$R>0$ in the three-dimensional  Euclidean space and its integral lines turn out
to be the meridians from the
north to the south pole.
  \end{remark}

\begin{remark}\label{PAN}
We study the limit of $\epsilon  {\mathcal M}$ when $\epsilon
\to 0$, in the northern hemisphere.

The
frame of left-invariant vector fields $\bar X = \epsilon X$, $\bar Y = \epsilon
Y$ and $\bar T =  \epsilon^{-2} T$ is independent of $\epsilon$. 
Moreover, 
 the linear
connection $\nabla$ restricted to the horizontal distribution spanned by $\bar X$ and $\bar Y$ is
independent of the parameter $\epsilon$. Indeed, from the fundamental relations
\eqref{FR} and from \eqref{tau-sigma} we find
\[
\begin{split}
 & \nabla_{\bar X}\bar X = \nabla_{\bar Y}\bar Y = 0,
\\
&
 \nabla_{\bar X} \bar Y = -\sigma \bar T\quad \textrm{and}\quad 
  \nabla_{\bar Y} \bar X = \sigma \bar T.
\end{split}
\]

Now, it turns out that 
 \[
 \begin{split}
 \bar{\mathcal M} = \lim_{\epsilon \to 0} \epsilon \mathcal M
&  = \frac 1 R \Big[ \Big(x\frac{\sqrt{R^2-r^2}}{r}-y\Big) \partial _ x
+\Big(y\frac{\sqrt{R^2-r^2}}{r} +x\Big) \partial _y - \sigma r^2\partial_t\Big]
\\
&
= (x\bar\lambda -y \bar \mu) \bar X + (y\bar \lambda+x\bar \mu) \bar Y,
 \end{split}
\]
where
\[
 \bar \lambda = \lambda = \frac{\sqrt{R^2-r^2}}{rR},
\qquad  
\bar\mu = \frac 1 R.
\]
The vector field $\bar {\mathcal M}$ is horizontal  and  tangent to the
Pansu's sphere.

We denote by   $J$  the complex structure  $J(\bar X) = \bar Y$ and
$J(\bar Y)
=- \bar X$.
A computation similar to the one in the proof of Theorem \ref{4.1}
shows that  
\begin{equation} \label{check}
 \nabla\! _{\bar{\mathcal M}} \bar {\mathcal M} =    \frac 2 R J (\bar {
\mathcal
M}).
\end{equation}
This is the equation for Carnot-Carath\'eodory geodesics in $H^1$ for the
sub-Riemannian metric making $\bar X$ and $\bar Y$ orthonormal, see
\cite[Proposition 3.1]{RR}. 

Thus, we reached the following conclusion. The 
  integral curves of $\mathcal M$ are Riemannian geodesics of $\Sigma_R$ and
converge
to the integral curves of $\bar{\mathcal M}$. These curves  foliate the Pansu's sphere
and are 
Carnot-Carath\'eodory geodesics (not only of the Pansu's sphere but also) of
$H^1$.

Using \eqref{check} we can pass to the limit as $\epsilon \to 0$ in
equation \eqref{espo}, properly
scaled.    An inspection of the right hand side in  \eqref{NUM1}  shows that
the right hand side of \eqref{espo} is asymptotic to $\epsilon ^4$.
In fact, starting  from \eqref{NUM2} 
we get
\begin{equation} \label{pipox}
  - \lim_{\epsilon\to0} \frac{H }{ \epsilon^{4}\omega (r)^2} \mathcal N = 
 \frac{1}{R\sigma^2 r^2} \big[- ( x\bar \mu +y\bar \lambda ) \bar X +(
x\bar\lambda -y\bar\mu ) \bar Y \big]  =  \frac{1}{R\sigma^2 r^2}
J(\bar{\mathcal M}).
\end{equation}
From \eqref{espo}, \eqref{check}, and \eqref{pipox} we deduce that
\[
  \lim_{\epsilon\to0} \epsilon^{-4} \nabla \!_{\mathcal M}\mathcal M =
\frac{1}{2\sigma^2 r^2}  
 \nabla\! _{\bar{\mathcal M}} \bar {\mathcal M}.
\]

\end{remark}

\section{Topological CMC spheres are left translations of $\Sigma_R$}

\label{k_0=0}
\setcounter{equation}{0}

 In this section, we prove that   any topological sphere  in $H^1$
having constant mean curvature is congruent to a sphere $\Sigma_R$ for some
$R>0$. This result was announced, in wider generality, in \cite{A5}. As in
\cite{AR}, our proof relies
on the identification of a holomorphic quadratic  differential for CMC
surfaces in $H^1$.

For an oriented  surface  $\Sigma$ in $H^1$ with unit normal vector $\enn$, we
denote by
$h\in \Hom(T_p\Sigma;T_p\Sigma)$  
the shape operator $h(W) = \nabla_W\enn$, at any point $p\in\Sigma$. 
The $1$-form $\vartheta$ in
$H^1$, defined by $\vartheta(W) = \langle W,T\rangle$ for $W\in \Gamma(T H^1)$, can be
restricted to the tangent bundle $T\Sigma$. The tensor product
$\vartheta\otimes\vartheta \in \Hom(T_p\Sigma; T_p\Sigma)$ is defined, as a linear
operator, by the
formula
\[
 (\vartheta\otimes\vartheta) (W) = \vartheta(W) (\vartheta(X_1) X_1+\vartheta(X_2)
X_2),\qquad W\in \Gamma( T\Sigma),
\]
where  $X_1,X_2$ is any (local) orthonormal frame of $T\Sigma$. Finally, for any
$H\in\R$ with $H\neq0$, let $\alpha_H\in(-\pi/4,\pi/4)$ be the angle
\begin{equation}
\label{alpha_H}
 \alpha_H= \frac 12 \arctan\Big(\frac{\tau}{H}\Big),
\end{equation}
and  let $q_H \in\Hom(T _p\Sigma; T_p \Sigma)$ be the (counterclockwise)
rotation by the angle
$\alpha_H$ of each tangent plane $T_p\Sigma$ with $p\in\Sigma$.

 \begin{definition}
  Let $\Sigma$ be an (immersed) surface in $H^1$ with constant mean curvature
$H\neq 0$. At any point $p\in\Sigma$,  we define the linear operator  $k\in
\Hom(T_p\Sigma; T_p\Sigma)$ 
by
\begin{equation} \label{ABR}
   k = h + \frac{2\tau^2}{\sqrt{H^2+\tau^2}} q_H \circ 
(\vartheta\otimes\vartheta)\circ q_H^{-1}.
\end{equation}
 \end{definition}

The operator $k$ is symmetric, i.e., $\langle k(V),W\rangle = \langle
V,k(W)\rangle$.
The trace-free part of $k$ is 
$ k _ 0 = k -\frac 12 \mathrm{tr}(k)\mathrm{Id}$. In  fact, we have 
\begin{equation} \label{b_0}
 k_ 0 = h_0 + \frac{2\tau^2}{\sqrt{H^2+\tau^2}} q_H  \circ 
(\vartheta\otimes\vartheta)_0\circ q_H^{-1}.
\end{equation}
Formula \eqref{ABR} is analogous to the formula for the quadratic holomorphic
differential discovered in \cite{AR}.

In the following, we identify the linear operators 
$h,k,\vartheta\otimes\vartheta$ with the corresponding bilinear forms
$(V,W)\mapsto h(V,W) = \langle h(V),W\rangle$, and so on.

The structure of $k$ in \eqref{ABR} can be established
in the following way. 
Let $\Sigma_R$ be the CMC sphere with $R = 1/ \varepsilon H$. 
From the formula \eqref{ACCA}, we deduce that,    in the frame
$X_1,X_2$   in \eqref{pipo}, the trace-free
shape operator   at the point $(z,t) \in\Sigma_R$  is given by
\begin{equation*}
\label{hst}
h_0   
= \frac{\varrho^2 }{1+\varrho^2}
\left(
\begin{array}{cc}
H & \tau
\\
\tau & - H
\end{array}
\right),
\end{equation*}
where $\varrho = \tau \e|z|$.
On the other hand, from \eqref{xixo} and \eqref{abcp},  we get
\[
  \vartheta(X_1)=0 \quad \textrm{ and } \quad  
\vartheta(X_2) =\frac{\varrho \sqrt{\tau^ 2+ H^ 2 } }{\tau\sqrt{
1+\varrho^2}},
\]
and we   therefore obtain the following formula for the trace-free tensor
$(\vartheta
\otimes\vartheta)_0$ in the frame $X_1,X_2$:
\[
(\vartheta\otimes\vartheta) _0 =  \displaystyle - \frac{(\tau^2+
H^2)}{2 \tau^2 }
\frac{\varrho^2
}{ 1+\varrho^2 }
\left(
\begin{array}{cc} 
1& 0
\\
0 & -1
\end{array}
\right).
\]
Now, in the unknowns  $c\in\R$ and $q$ (that is a rotation by  an angle $\beta$), the system
of equations $h_0 + c q 
(\vartheta\otimes\vartheta)_0 q ^{-1}=0$
holds independently of $\varrho$  if and only if  
$c= 2\tau^2/\sqrt{H^2+\tau^2}$ and
$\beta$ is  the
angle in \eqref{alpha_H}.  
We record this fact in the next:

\begin{proposition} \label{k_0=0_for_S_R}
The linear operator $k$   on  the sphere   $\Sigma_R$ with mean curvature
$H$, at the point $(z,t) \in \Sigma_R$, is given by
\[
 k  = 
\Big( H + \frac{\varrho^2}{1+\varrho^2} \sqrt{\tau^2 + H^2} \Big)
  \mathrm{Id}.
\]
In particular, $\Sigma_R$ has  vanishing  $k_0$ (i.e., $k_0=0$).
\end{proposition}

In Theorem \ref{5.5}, we prove that \emph{any} topological sphere in $H^1$ with
constant mean curvature has
vanishing
  $k_0$. We need to work in a conformal frame of tangent vector
fields to the
surface.

Let $z=x_1 + i x_2$ be the complex variable.
Let $D\subset\C$ be an open set and, for a given map  $F\in C^\infty(D;H^1)$,
consider
the immersed   surface  $\Sigma=F(D)\subset\H^1$.
The parametrization   $F$ is  conformal   if
there exists a positive function $E\in C^\infty(D)$ 
such that, at any point in $D$,    the vector fields $ V_1=F_*  \frac
{\partial}{   \partial x_1}$  and
$V_2=F_* \frac{\partial}{\partial x_2}$ satisfy:
\begin{equation}
\label{eq:conf}
| V_1|^2=| V_2|^2=E,\quad 
\la V_1,V_2\ra=0.
\end{equation} 
We call $V_1,V_2$ a conformal frame for $\Sigma$ and 
we denote by $\enn $ the  normal vector field  to $\Sigma$ such that triple
$V_1,V_2,\enn$ forms a  positively oriented frame, i.e.,
\begin{equation}
 \label{PO}
    \enn = \frac{1}{E} V_1\wedge V_2.
\end{equation}

The  second fundamental form of $\Sigma$ in the frame
$V_1,V_2$ is denoted by 
\begin{equation}
\label{eq:2ff}
h=
 (h_{ij})_{i,j=1,2}
=\begin{pmatrix}
L&M\\
M&N
\end{pmatrix},\quad h_{ij}=\la \nabla_i\enn ,V_j\ra,
\end{equation}
where $\nabla_i=\nabla_{V_i}$ for $i=1,2$. This notation differs from \eqref{ACCA},
where the fixed frame is $X_1,X_2,\enn$.
Finally,
 the {\em mean curvature} of $\Sigma$
is
\begin{equation}
\label{eq:H}
H=\frac{L+N}{2E}=\frac{h_{11}+h_{22}}{2E}.
\end{equation}

By Hopf's technique on holomorphic quadratic
differentials, 
the validity of the equation $k_0=0$ follows
from the   Codazzi's equations, which  involve
curvature terms. An interesting relation
between
 the $1$-form $\vartheta$
and  the Riemann curvature
operator, defined as   
\[
   R(U,V)W = \nabla_U\nabla_V W - \nabla_V\nabla_U W-\nabla_{[U,V]} W
\]
for any  $U,V,W \in  \Gamma(T H^1)$,   is described 
in the following:

\begin{lemma}
\label{lem:van11}
Let $V_1,V_2$ be a conformal frame of an immersed surface $\Sigma$ in $H^1$
with conformal factor $E$ and unit normal $\enn$. Then, we have  
\begin{align}
\label{eq:c2a}
\la R(V_2,V_1)\enn ,V_2\ra= 
                   4\tau^2 E\theta(V_1)\theta(\enn).
\end{align}
\end{lemma}
\begin{proof} 
We use   the notation 
\begin{equation}
 \label{viv}
\begin{split}
 V_i & =V_i^XX+V_i^YY+V_i^TT,\qquad  i=1,2,
\\
\enn & =\enn^XX+\enn^YY+\enn^TT.
\end{split}
\end{equation}
From the fundamental relations \eqref{FR}, we obtain:
\[
\begin{array}{lllr}
\la R(V_2,V_1)\enn ,V_2\ra
& = & V_2^XV_1^Y\enn^YV_2^X\cdot(-3\tau^2)&(1)\\
&& + V_2^XV_1^Y\enn^XV_2^Y\cdot(3\tau^2)&(2)\\
&& + V_2^XV_1^T\enn^TV_2^X\cdot(\tau^2)&(3)\\
&& + V_2^XV_1^T\enn^XV_2^T\cdot(-\tau^2)&(4)\\
&& + V_2^YV_1^X\enn^XV_2^Y\cdot(-3\tau^2)&(5)\\
&& + V_2^YV_1^X\enn^YV_2^X\cdot(3\tau^2)&(6)\\
&& + V_2^YV_1^T\enn^TV_2^Y\cdot(\tau^2)&(7)\\
&& + V_2^YV_1^T\enn^YV_2^T\cdot(-\tau^2)&(8)\\
&& + V_2^TV_1^X\enn^XV_2^T\cdot(\tau^2)&(9)\\
&& + V_2^TV_1^X\enn^TV_2^X\cdot(-\tau^2)&(10)\\
&& + V_2^TV_1^Y\enn^YV_2^T\cdot(\tau^2)&(11)\\
&& + V_2^TV_1^Y\enn^TV_2^Y\cdot(-\tau^2).&(12)
\end{array}
\]
Now, we have  $(9)+(10)+(11)+(12)=0$. In fact:
\[\begin{split}
(9)+(11)&=\tau^2V_2^TV_2^T(V_1^X\enn ^X+V_1^Y\enn ^Y)=-\tau^ 2
V_2^TV_2^TV_1^T\enn
^T,\\
(10)+(12)&=-\tau^2V_2^T\enn ^T(V_1^XV_2^X+V_1^YV_2^Y)=\tau^2V_2^T\enn
^TV_1^TV_2^T,
\end{split}
\]
where we used $\la V_1,\enn \ra=\la V_1,V_2\ra=0$ to deduce
$V_1^X\enn ^X+V_1^Y\enn ^Y=-V_1^T\enn ^T$ and
$V_1^XV_2^X+V_1^YV_2^Y=-V_1^TV_2^T$.
Moreover,  we have $(3)+(4)+(7)+(8)=\tau^2 E V_1^T\enn ^T$. Indeed, 
\[
\begin{split}
(3)+(7)&=\tau^2 V_1^T\enn ^T(V_2^XV_2^X+V_2^YV_2^Y)=\tau^2V_1^T\enn
^T(E-V_2^TV_2^T),\\
(4)+(8)&=-\tau^2V_1^TV_2^T(V_2^X\enn ^X+V_2^Y\enn ^Y)=\tau^2V_1^TV_2^TV_2^T\enn
^T,
\end{split}
\]
where we used $\la V_2,V_2\ra=E$ and $\la V_2,\enn \ra=0$ to deduce
$V_2^XV_2^X+V_2^YV_2^Y=E-V_2^TV_2^T$ and $V_2^X\enn ^X+V_2^Y\enn ^Y=-V_2^T\enn
^T$. Indeed, 
\[
\begin{split}
(1)+(5)&=-3\tau^2(V_2^XV_1^Y\enn ^YV_2^X+V_2^YV_1^X\enn ^XV_2^Y)\\
&=3\tau^2[V_1^T\enn ^T(V_2^XV_2^X+V_2^YV_2^Y)+V_2^XV_1^X\enn
^XV_2^X+V_2^YV_1^Y\enn ^YV_2^Y]\\
&=3\tau^2[V_1^T\enn ^T(E-V_2^TV_2^T)+V_2^XV_1^X\enn ^XV_2^X+V_2^YV_1^Y\enn
^YV_2^Y]\\
(2)+(6)&=3\tau^2[V_2^XV_1^Y\enn ^XV_2^Y+V_2^YV_1^X\enn ^YV_2^X]\\
&=-3\tau^2[V_1^TV_2^T(V_2^X\enn ^X+V_2^Y\enn ^Y)+V_2^XV_1^X\enn
^XV_2^X+V_2^YV_1^Y\enn ^YV_2^Y]\\
&=-3\tau^2[-V_1^TV_2^T\enn ^TV_2^T+V_2^XV_1^X\enn ^XV_2^X+V_2^YV_1^Y\enn
^YV_2^Y],
\end{split}
\]
where we used $\la V_1,\enn \ra=\la V_1,V_2\ra=0$ to deduce 
$V_1^Y\enn ^Y=-V_1^X\enn ^X-V_1^T\enn ^T$, 
$V_1^YV_2^Y=-V_1^XV_2^X-V_1^TV_2^T$ 
and $V_1^XV_2^X=-V_1^XV_2^X-V_1^TV_2^T$.
Equation \eqref{eq:c2a}  follows.
 
\end{proof}

 For an  immersed   surface  with conformal frame  $V_1,V_2$, 
we use   the notation 
$V_iE=E_i$, $V_iH=H_i$, $V_iN=N_i$, $V_iM=M_i$, $V_iL=L_i$, $i=1,2$.

\begin{theorem}[Codazzi's Equations]
Let $\Sigma=F(D)$ be an immersed  surface in $\H^1$ 
with conformal frame $V_1,V_2$, conformal factor $E$ and unit normal $\enn$.
Then, we have 
\begin{align}
\label{eq:coda1}
H_1&=\frac{1}{E}\Big\{
\frac{L_1-N_1}{2}+M_2-4\tau^2E\theta(V_1)\theta(\enn)\Big\},\\
\label{eq:coda2}
H_2&=\frac{1}{E}\Big\{
\frac{N_2-L_2}{2}+M_1-4\tau^2E\theta(V_2)\theta(\enn)\Big\},
\end{align}
where $L,M,N, H$ are as in 
\eqref{eq:2ff} and \eqref{eq:H}.
\end{theorem}

\begin{proof}
We start from the following well-known formulas for the derivatives of the mean
curvature: 
\begin{align}
\label{eq:cod1}
H_1&=\frac{1}{E}\Big\{ \frac{L_1-N_1}{2}+M_2+\la R(V_1,V_2)\enn , V_2\ra
\Big\},\\
\label{eq:cod2}
H_2&=\frac{1}{E}\Big\{ \frac{N_2-L_2}{2}+M_1+\la R(V_2,V_1)\enn , V_1\ra \Big\}.
\end{align}
Our claims \eqref{eq:coda1} and \eqref{eq:coda2}
follow from these formulas and  Lemma \ref{lem:van11}. 

For the reader's convenience, 
we give a   short sketch of the proof of \eqref{eq:cod1},
see e.g.~\cite{Kling}  for the flat case.  
For
any $i,j,k=1,2$, we have
\begin{equation}\label{eq:c1}
  V_k h_{ij}-V_i h_{kj} = \la R(V_k,V_i)\enn ,V_j\ra+\la\nabla_i\enn
,\nabla_kV_j\ra-\la\nabla_k\enn ,\nabla_iV_j\ra.
\end{equation} 
Setting   $i=j=2$  and $k=1$ in  \eqref{eq:c1}, and using \eqref{eq:H} we 
  obtain
\begin{equation}
\label{eq:c12}\begin{split}
V_1(2EH)
&=L_1+M_2+\la R(V_1,V_2)\enn ,V_2\ra
+\la\nabla_2\enn ,\nabla_1V_2\ra
-\la\nabla_1\enn ,\nabla_2V_2\ra.
\end{split}
\end{equation}
Using   the expression of $\nabla_i\enn$ in the conformal frame, we find   
\begin{equation}
\label{eq:c13} \la\nabla_2\enn ,\nabla_1V_2\ra
-\la\nabla_1\enn ,\nabla_2V_2\ra =HE_1,
\end{equation}
and from  \eqref{eq:c12} and \eqref{eq:c13} we deduce  that 
\begin{equation} \label{eq:c14}
H_1=\frac{1}{2E}\{L_1-E_1H+M_2+\la R(V_1,V_2)\enn ,V_2\ra\}.
\end{equation}
From \eqref{eq:H}, we have the further equation
\[
L_1-E_1H =\frac{L_1-N_1}{2}+EH_1,
\]
that, inserted into \eqref{eq:c14},
gives claim  \eqref{eq:cod1}.

\end{proof}

Now we switch to the complex variable $z = x_1+ i x_2 \in D$ and define the
complex
vector fields 
\[
\begin{split}
  Z = \frac 12 (V_1-i V_2)  = F_* \Big( \frac{\partial  }{\partial z}\Big),
  \\
  \bar Z = \frac 12 (V_1+i V_2)  = F_* \Big( \frac{\partial  }{\partial \bar
z}\Big).
\end{split}
\]
Equations \eqref{eq:coda1}-\eqref{eq:coda2}
can be transformed into one single equation:
\begin{equation} \label{COMPO}
E(ZH) = \bar Z\Big(\frac{L-N}{2} - i M\Big) - 4\tau^2 E\vartheta(\enn)\vartheta(Z). 
\end{equation}

 Now consider the trace-free part of $b=k-h$, i.e.,
 \begin{equation*}
 \label{stio}
    b_0 =  \frac{2\tau^2}{\sqrt{H^2+\tau^2}} q_H \circ 
(\vartheta\otimes\vartheta)_0\circ q_H^{-1}
 \end{equation*}
The entries of $b_0$ as a quadratic form in the conformal frame $V_1,V_2$, with
$\vartheta_i = \vartheta(V_i)$ and $
c_H =\frac{2\tau^2}{H^2+\tau^2}$, are given by
\begin{equation} \label{poppo}
\begin{split}
A & =b_0(V_1,V_1) = c_H\Big( H\frac{\theta_1^2-\theta_2^2}{2}-\tau\theta_1\theta_2\Big),
\\
B & = b_0(V_1,V_2) = c_H\Big(H\theta_1\theta_2+\tau\frac{\theta_1^2-\theta_2^2}{2}\Big).
\end{split}
\end{equation}
These entries can be computed starting from $q_H (\vartheta\otimes\vartheta)_0 q_H^{-1} = q_H^2  (\vartheta\otimes\vartheta)_0$, where $q_H^2$ is the rotation by the angle $2\alpha_H$ that, by \eqref{alpha_H}, satisfies $\cos(2\alpha_H)=H/\sqrt{H^2+\tau^2}$ and $\sin(2\alpha_H)=\tau/\sqrt{H^2+\tau^2}$.

\begin{lemma} 
\label{lem:van12}
Let  $\Sigma $ be an immersed surface in $H^1$ with constant mean curvature $H$
and unit normal $\enn$ such that $V_1,V_2,\enn$ is positively oriented. Then,  
on $\Sigma$ we have
\begin{equation} 
\label{stix}
\bar Z (A-iB) = - 4 \tau^2 E \vartheta(\enn)\vartheta(Z).
\end{equation}
\end{lemma}

\begin{proof} 
The complex equation \eqref{stix} is equivalent to the system
of real equations  
\begin{equation}
\label{eq:c3a}
\begin{split}
A_1+ B_2 &= -4\tau^2 E \vartheta(\enn) \vartheta(V_1),
\\ A_2-B_1&= 4\tau^2  E \vartheta(\enn) \vartheta(V_2),
\end{split}
\end{equation}
where $A_i = V_i A$ and $B_i = V_i B$, $i=1,2$.

We check the first equation in \eqref{eq:c3a}. Since $H$ is constant, we
have 
\[
 A_1 + B_2 = c_H 
H\Big\{V_1\Big(\frac{\theta_1^2-\theta_2^2}{2}\Big)+V_2(\theta_1\theta_2)\Big\}
+\tau c_H 
\Big\{V_2\Big(\frac{\theta_1^2-\theta_2^2}{2}\Big)-V_1(\theta_1\theta_2)\Big\},
\]
where
\[
\begin{split}
 V_1\Big(\frac{\theta_1^2-\theta_2^2}{2}\Big)+V_2(\theta_1\theta_2)
& = \theta_1  (V_1\theta_1+V_2\theta_2)+\theta_2 
(V_2\theta_1-V_1\theta_2),
\\V_2\Big(\frac{\theta_1^2-\theta_2^2}{2}\Big) 
-
V_1(\theta_1\theta_2) 
& = \theta_1 
(V_2\theta_1-V_1\theta_2)-\theta_2 (V_1\theta_1+V_2\theta_2).
\end{split}
\]
For $i,j=1,2$, we have  
\begin{equation}
\label{eq:Vtheta}
V_i\theta_j= \la\nabla_iT,V_j\ra+\la T,\nabla_iV_j\ra,
\end{equation}
where, with the notation \eqref{viv} and    by the fundamental relations 
\eqref{FR},
\begin{equation} \label{VALX}
\la  \nabla_iT, V_j\ra = \la \tau V_i^XY-\tau V_i^YX, V_j\ra = \tau
V_i^XV_j^Y-\tau V_i^YV_j^X.
\end{equation}
From \eqref{eq:Vtheta}, \eqref{VALX}, \eqref{PO}, and 
\[
\nabla_2 V_1 - \nabla_1
V_2 =[V_2,V_1] = \Big[F_* \frac{\partial}{\partial
x_2},F_*\frac{\partial}{\partial x_1}\Big]= F_*\Big[\frac{\partial}{\partial
x_2},\frac{\partial}{\partial x_1}\Big]=0 ,
\]
we deduce 
\begin{equation}
\begin{split}
\label{eq:a2}
   V_2\theta_1-V_1\theta_2&=2\tau(V_1^YV_2^X-V_1^XV_2^Y)
 + \la T,\nabla_2V_1-\nabla_1V_2\ra=-2\tau E\theta(\enn).
\end{split}
\end{equation}

By the definition \eqref{eq:H} and \eqref{eq:conf}, we have 
\[ 
\nabla_1V_1
+ \nabla_2V_2 =\la\nabla_1V_1
+ \nabla_2V_2  ,\enn \ra\enn =-2EH\enn,
\]
and thus, again from \eqref{eq:Vtheta} and \eqref{VALX}, we obtain 
\begin{equation}
\label{eq:a1}
 V_1\theta_1+V_2\theta_2
 =\theta(\nabla_1V_1+\nabla_2V_2)= -2EH\vartheta(\enn).
\end{equation}

From
\eqref{eq:a1} and \eqref{eq:a2} we deduce that  
\begin{equation}
\label{eq:aa}
\begin{split}
  V_1\Big(\frac{\theta_1^2-\theta_2^2}{2}\Big)+V_2(\theta_1\theta_2)
 &
 =
-2E\theta(\enn)[H\theta(V_1)+\tau\theta(V_2)],
\\
V_2\Big(\frac{\theta_1^2-\theta_2^2}{2}
\Big)- V_1(\theta_1\theta_2) & =- 2E\theta(\enn)[\tau
\theta(V_1)-H\theta(V_2)],
\end{split}
\end{equation}
and finally 
\[
\begin{split}
A_1+B_2&= -2 c_H (H^2+\tau^2) E\theta(\enn)\theta(V_1) = -4\tau^2
E\vartheta(\enn)\vartheta(V_1).
\end{split}
\]

In order to prove the second equation in \eqref{eq:c3a}, notice that
\[
B_1-A_2=c_HH\Big\{V_1(\theta_1\theta_2)-V_2\Big(\frac{\theta_1^2-\theta_2^2}{2}
\Big)\Big\}
+c_H\tau\Big\{V_2(\theta_1\theta_2)+V_1\Big(\frac{\theta_1^2-\theta_2^2}{2}
\Big)\Big\}.
\]
By \eqref{eq:aa} we hence obtain
\[\begin{split}
B_1-A_2&=c_HH\Big\{2E\theta(\enn)[\tau
\theta(V_1)-H\theta(V_2)]\Big\}
-c_H\tau\Big\{2E\theta(\enn)[H\theta(V_1)+\tau\theta(V_2)]\Big\}\\
&=-2c_H(H^2+\tau^2)E\theta(\mathcal N)\theta(V_1)= -4\tau^2
E\vartheta(\enn)\vartheta(V_2).
\end{split}
\] 
\end{proof}

Let $\Sigma$ be an immersed surface in $H^1$ defined in terms of a 
conformal
parametrization $F\in C^\infty(D;H^1)$.  
Let $f\in C^\infty(D;\C)$ be the function of the complex variable $z\in D$
given by
\begin{equation} \label{poi}
 f(z) = \frac{L-N}{2} - i M + A-i B,
\end{equation}
where $L,M,M,A,B$ are defined as in \eqref{eq:2ff} and \eqref{poppo} via the
conformal frame $V_1,V_2$ and are evaluated  at the
point $F(z)$.

\begin{proposition} 
\label{prop:van1}
If $\Sigma$ has constant mean curvature $H$ then the function $f$ in \eqref{poi}
is holomorphic in $D$.
\end{proposition}

\begin{proof}
 From \eqref{COMPO} with $ZH=0$ and \eqref{stix}, we obtain the equation
on $\Sigma = F(D)$
\[
 \bar Z\Big(\frac{L-N}{2} - i M + A -iB\Big) = 0,
\]
that is equivalent to $\partial _{\bar z} f =0$ in $D$. 
\end{proof}

Now, by a standard argument of Hopf, see \cite{H} Chapter VI, for topological
spheres the
function $f$ is identically zero.  
By Liouville's theorem, this follows from the estimate
\[
 |f(z)| \leq \frac{C}{|z|^4},\quad z\in\C,
\]
that can be  obtained expressing the second fundamental forms in two different
charts without the north and south pole, respectively. We skip the details
of the proof of the next:

\begin{theorem}
\label{5.5}
A topological sphere  $\Sigma$  immersed in $H^1$  with constant mean
curvature has vanishing  $k_0$.
\end{theorem}

In the rest of this section, we show how to deduce from the equation $k_0=0$
that any topological sphere is congruent to a sphere $\Sigma_R$.
Differently from \cite{AR}, we do not use the fact that the isometry group of
$H^1$ is four-dimensional.

Let $\mathfrak h$ be the Lie algebra of $H^1$ and let 
$\langle\cdot,\cdot\rangle$ be the scalar product making $X,Y,T$ orthonormal.  
We denote by $S^2 = \{\nu\in\mathfrak h: |\nu| =
\sqrt{\langle\nu,\nu\rangle}=1\}$ the unit
sphere in $\mathfrak h$. 
For any $p\in H^1$, let $\tau^p: H^1\to H^1$ be the left-translation
$\tau^p (q) = p^{-1}\cdot q$ by the inverse of $p$, where $\cdot$ is the group
law of $H^1$, and
denote by $\tau_*^p\in \Hom(T_p H^1; 
\mathfrak h)$ its differential. 

For any   point $(p,\nu) \in H^1\times S^2$ there is
a unique   $\enn \in T_p H^1$ such that   $\nu = \tau^p_*\enn
$ and we define  $T^\nu_p H^1 = \{ W \in T_p H^1
: \langle W,\enn\rangle = 0 \}$. Depending on the point $(p,\nu)$ and on the
parameters $H,\tau \in\R$, with $H^2+\tau^2\neq 0$, below we define the linear
operator $\L_{H}\in \Hom( T_p^\nu H^1;  T_\nu S^2)$. The definition is
motivated by the proof of Proposition \ref{badpete}. For any
$W\in T_p^\nu M$, we let 
\[
 \L_{H} W = \tau_*^p\Big(H W -\frac{2\tau^2}{\sqrt{H^2+\tau^2}}
\q_{H}  ( \vartheta\otimes\vartheta )_0 \q_H^{-1} W\Big) +
(\nabla_W\tau_*^p)(\enn),
\]
where $\nabla_W\tau_*^p\in \Hom(T_p H^1; \mathfrak h)$ is the  covariant
derivative of $\tau_*^p$ in the direction $W$
and  the
trace-free operator $ ( \vartheta\otimes\vartheta )_0 \in\Hom(
T^\nu_p H^1; T^\nu_p H^1)$ is 
\[
 ( \vartheta\otimes\vartheta )_0 =  \vartheta\otimes\vartheta -\frac 1 2 
\mathrm{tr}(\vartheta\otimes\vartheta )\mathrm{Id}.
\] 
The operator $q_{H}\in\Hom(T^\nu_p H^1 ;  T^\nu_p H^1)$ is
the rotation by the angle $\alpha_H$ in  \eqref{alpha_H}.  
The operator $
 \L_{H } $ is well-defined, i.e., $\L_{H } W \in\mathfrak h$ and 
$\langle \L_{H } W,\nu\rangle =0$ for any $W \in T_p^\nu H^1$. This can be
checked using the identity $|\enn|=1$ and working with the formula
\[
(\nabla _ W\tau_*^p)(\enn) = \sum_{i=1}^3 \langle \enn, \nabla_W Y_i\rangle
Y_i(0),
\]
where $Y_1,Y_2,Y_3$ is any frame of orthonormal left-invariant vector fields.

Finally, for any point $(p,\nu) \in H^1\times S^2$, define 
\[
 \mathcal E_{H }(p,\nu) = \big\{ (W, \mathcal L_{H} W) : W\in T_p^\nu
H^1\big\}\subset T_{p}  H^1\times T_{ \nu  } S^2 .
\]
Then $(p,\nu)\mapsto 
 \mathcal E_{H }(p,\nu)$ is a distribution of two-dimensional planes in
$H^1\times S^2$.
The distribution $\mathcal E_{H }$ origins from CMC surfaces with mean
curvature $H$ and vanishing $k_0$.

Let $\Sigma$ be a smooth oriented   surface immersed in $H^1$  
given by a parameterization  $F\in C^\infty(D; H^1)$ 
where $D\subset\C$ is an open set.
We denote by $\enn(F(z)) \in T_p H^1$, with $p=F(z)$, the unit normal of
$\Sigma$ at the point $z \in D$.   
The normal section is given 
by the mapping $G: D\to S^2$ defined by $G(z) =
\tau_*^{F(z)} \enn(F(z))$, and we can define the Gauss section $\Phi:D \to
H^1\times S^2$ letting $\Phi(z) = (F(z), G(z))$.
Then $\overline{\Sigma} = \Phi(D)$ is a two-dimensional immersed surface in
$H^1\times
S^2$, called the \emph{Gauss extension} of $\Sigma$.

\begin{proposition}\label{badpete}
 Let $\Sigma$ be an oriented   surface immersed in $H^1$ with
constant  mean curvature $H$ and vanishing   $k_0$. Then the
Gauss extension $\overline{\Sigma}$ is an
integral surface of the distribution $\mathcal E_{H}$ in $H^1\times S^2$.
\end{proposition}

\begin{proof} Let $\enn$ be the unit normal to $\Sigma$. For any tangent section
$W\in \Gamma( T\Sigma)$, we have 
\[
\begin{split}
 W ( \tau^F_* ( \enn  ) ) & =  \tau^F_*
(\nabla_W \enn  ) + (\nabla_W \tau_*^F)(\enn)
\\
&
      = \tau^F_* ( h(W)) + (\nabla_W \tau_*^F)(\enn),
\end{split}
\]
where $h(W) = \nabla_W\enn$ is the shape operator. Therefore,  the
set of all sections of the 
tangent bundle of $\overline{\Sigma}$   is 
\[
\Gamma(  T\overline{\Sigma})  = \Big\{\Big( W, \tau ^F_*( h(W) )+(\nabla_W
\tau_*^F) (\enn) \Big)   :  W \in
\Gamma(T\Sigma)  \Big\}. 
\]

The  equation $ k_0=0$ is equivalent to
$ h = H\mathrm{Id} -b_0$ where, by \eqref{b_0},
\[
b_0 =\frac{2 \tau^2}{\sqrt{H^2+\tau^2}} q_{H } 
\Big(\vartheta\otimes\vartheta -
\frac{\mathrm{tr}(\vartheta\otimes\vartheta)}{2} \mathrm{Id}\Big)q_{H } ^{-1},
\]
and thus the sections of  $\overline{\Sigma}$ are of the form
\[
 (W,\mathcal L_{H} W) \in \Gamma(T\overline{\Sigma})\quad
 \textrm{with}\quad W \in \Gamma(T\Sigma).
\]
This concludes the proof.
\end{proof}

\begin{theorem}\label{5.9}
 Let 
  $\Sigma$  be a topological sphere in $ H^1$
with constant mean curvature $H$. Then  there exist a left translation  
$\iota $ and
$R>0$ such that $\iota(\Sigma) = \Sigma_R$.

\end{theorem}

\begin{proof} 
Let $H>0$   be the mean curvature of $\Sigma$, let
$R=1/ H\varepsilon$, and recall that the sphere $\Sigma_R$ has mean curvature
$H$. 

Let $T^\Sigma(p) \in T_p\Sigma$ be the orthogonal projection
of the vertical vector field  $T$ onto $T_p\Sigma$.  Since $\Sigma$ is a
topological
sphere,
there exists a point $p\in\Sigma$ such that $T^\Sigma(p) = 0$. This implies
that either  $T=\enn$
or $T=-\enn$ at the point $p$, where $\enn$ is the outer normal to $\Sigma$ at
$p$. Assume that $T = \enn$.

Let $\iota   $ be the left translation such that $\iota(p)
= N$, where $N$ is the north pole of $\Sigma_R$. At the point $N$ the vector $T
$ is  the outer normal to $\Sigma_R$. Since $\iota_* T = T$ (this holds for
any isometry),  we deduce that $\Sigma_R$ and $\iota(\Sigma)$ are two surfaces
such that:
\begin{itemize}
 \item [i)] They have both constant mean curvature $H$.
 \item [ii)] They have both vanishing   $k_0$, by Proposition 
\ref{k_0=0_for_S_R} and
Theorem \ref{5.5}.
 \item [iii)] $N\in \Sigma_R\cap \iota(\Sigma)$ with the
same (outer) normal at $N$.
\end{itemize}
Let $M_1=\overline{\Sigma}_R$ and $M_2=\overline{\iota(\Sigma)}$ be the Gauss
extensions of $\Sigma_R$ and $\iota(\Sigma)$, respectively. Let $\nu
=\tau^N_*\enn \in S^2$.
From i), ii) and Proposition \ref{badpete} it follows that $M_1$ and $M_2$ are
both integral surfaces of the distribution $\mathcal E_{H}$.
From iii), it follows that $(N,\nu)\in M_1\cap M_2$. Being the two surfaces
complete, this implies that
$M_1=M_2$ and thus $\Sigma_R = \iota(\Sigma)$. 

\end{proof}

\section{Quantitative stability of $\Sigma_R$ in vertical cylinders}
\setcounter{equation}{0}
  \label{SEI}

In this section, we prove a quantitative isoperimetric inequality for the CMC
spheres $\Sigma_R$ with respect to compact perturbations in vertical cylinders,
see Theorem \ref{thm:quant}.
This is a strong form of stability of $\Sigma_R$ in the northern  and
southern hemispheres.

A CMC surface $\Sigma$ in $H^1$ with normal $\enn$ is stable in an open region
$A\subset\Sigma$ if for any function $g \in C^\infty_c(A)$ with $\int_\Sigma g
d\A=0$,
where $\A$ is the Riemannian area measure of $\Sigma$, we have
\[
 \mathcal S(g) = \int_\Sigma \big\{|\nabla  g |^2 - 
\big(|h|^2+\Ric(\enn)\big)g^2\big\} d\A \geq 0. 
\]
The functional $\mathcal S(g)$ is the second variation, with fixed volume, of
the area of $\Sigma$ with respect to the infinitesimal deformation of $\Sigma$
in the direction $g\enn$. Above, $|\nabla   g|$ is the length of the
tangential gradient of $g$,  $|h|^2$ is the squared norm of the second
fundamental form of $\Sigma$ and $\Ric(\enn)$ is the Ricci curvature of $H^1$ in
the direction $\enn$.

The Jacobi operator associated with the second variation functional $\mathcal S$
is
\begin{equation*} \label{Jacobi}
 \mathcal L g = \Delta   g + (|h|^2 +\Ric(\enn))g   ,
\end{equation*}
where $\Delta $ is the Laplace-Beltrami operator of $\Sigma$. 
As a consequence of Theorem 1 in \cite{FCS}, if there exists a strictly
positive  solution 
 $g\in C^\infty(A)$ to equation $\mathcal L g =0$  on $A$,  then
$\Sigma$ is stable in $A$ (even without
the restriction $\int_A gd\mathcal A=0$).

Now consider in $H^1$ the right-invariant vector fields
\begin{equation*}\label{XYTright}
\widehat X  = \frac 1 \epsilon\Big( \frac{\partial}{\partial x}-\sigma y
\frac{\partial}{\partial
t}\Big),
\quad
   \widehat Y  = \frac 1 \epsilon\Big( \frac{\partial}{\partial y}  + \sigma
x\frac{\partial}{\partial
t}\Big),
\quad\textrm{and}\quad 
   \widehat   T  = \epsilon^2 \frac{\partial}{\partial t}.
\end{equation*}
These are generators of left-translations in $H^1$, and the functions
\[
 g_{\widehat X} = \langle \widehat X,\enn\rangle,\quad
 g_{\widehat Y} = \langle \widehat Y,\enn\rangle,\quad 
 g_{\widehat T} = \langle \widehat T,\enn\rangle 
\]
are solutions to $\mathcal L g=0$. By the previous
discussion, the CMC sphere
$\Sigma_R$ is stable in the hemispheres
\[
\begin{split}
 A _{ {\widehat X}} &  =  \big\{ (z,t) \in \Sigma_R : g_{\widehat
X}>0\big\},
\\
 A _{ {\widehat Y}} &  =  \big\{ (z,t) \in \Sigma_R : g_{\widehat
Y}>0\big\},
\\
 A _{ {\widehat T}} &  =  \big\{ (z,t) \in \Sigma_R : g_{\widehat
T}>0\big\}.
\end{split}
\]
In particular, $\Sigma_R$ is stable in the northern hemisphere  $A _{ {\widehat
T}}
= \{ (z,t) \in \Sigma_R : t>0\}$. 

In fact, we believe that the whole $\Sigma_R$ is stable. Actually, this would
follow from
the isoperimetric property for $\Sigma_R$. The proof of the stability of
$\Sigma_R$ requires a deeper analysis and it is not yet clear.
However, in the case of the northern (or southern) hemisphere we can prove a
strong
form of stability in terms of a quantitative isoperimetric inequality.
Some stability results in various sub-Riemannian settings have been recently
obtained in
\cite{Mo,HR1,HR2}.

For $R>0$, let $E_R\subset H^1$ be the open domain bounded by the CMC sphere
$\Sigma_R$, 
\[
        E_R=\{(z,t)\in\H^1 : |t|<f(|z|;R),\ |z|<R\},
\]  
where $f(\cdot;R)$ is the profile function of $\Sigma_R$ in \eqref{fuf}.
For   $0 \leq \delta<R$, we define the half-cylinder
\[
    C_{R,\delta}=\{(z,t)\in\H^1 : |z|<R\text{ and }t>  t_{R,\delta}\},
\]
where $t_{R,\delta}=f(r_{R,\delta};R)$  and $r_{R,\delta}=R-\delta$.
In the following, we use the   short notation  
\begin{equation}
\begin{split}
\label{eq:k}
k_{R\varepsilon\tau} & = {\varepsilon^3\omega(R)\sqrt{R}}, 
\\
C_{R\varepsilon\tau} & = 
\frac{1}{4\pi\varepsilon
R^3(Rk_{R\varepsilon\tau} +f(0; R))},
\\
D_{R\varepsilon\tau} & = \frac{1}{12\varepsilon\pi^2
R^5(4Rk_{R\varepsilon\tau}^2+  f(0;R)^2)} .
\end{split}
\end{equation}

We denote by $\A$ the Riemannian surface-area measure in
$H^1$.

\begin{theorem}   \label{thm:quant}  
Let $R>0$, $0\leq \delta<R$, $\epsilon>0$,  
and $\tau \in\R$ be as in  \eqref{tau-sigma}.
Let   $E\subset \H^1$ be a smooth open set such that   
$\mathcal L^3(E) =\mathcal L^3( E_R)$ and  $\Sigma = \partial E$.
\begin{itemize}
 \item [(i)] If $E\Delta E_R \subset\subset C_{R,\delta}$ with
  $0<\delta<R$ then  we have
\begin{equation}
\label{TP2}
\A(\Sigma  ) -\A( \Sigma_R)  \geq \sqrt{\delta} 
C_{R\varepsilon\tau} \mathcal L^{3}(E\Delta E_R)^2.
\end{equation}
 \item [(ii)] If 
$E\Delta E_R \subset\subset C_{R,0}$ then  we have
\begin{equation}
\label{TP}
\A(\Sigma) - \A(\Sigma_R)   \geq  D_{R\varepsilon\tau}  \mathcal
L^{3}(E\Delta E_R)^3.
\end{equation}
\end{itemize}
\end{theorem}

\begin{remark}\label{6.2}
When $\Sigma\subset H^1$ is a $t$-graph, $\Sigma = \{ (z,f(z)) \in H^1: z\in
D\}$ for some $f\in C^1(D)$, from \eqref{AER1} and \eqref{AER2} we see that 
the Riemannian area of $\Sigma $ is
\[
\mathcal A(\Sigma) =   \frac 1 \epsilon  \int_{D }    
\sqrt {\epsilon ^6+  
|\nabla f| ^ 2 + \sigma^2 |z|^2+2\sigma(xf_y-y f_x)}\ dz ,
\]
and so
\[
 \lim_{\epsilon\to 0}  \epsilon \mathcal A(\Sigma) =    \int_{D }  
\sqrt { 
|\nabla f| ^ 2 + \sigma^2 |z|^2+2\sigma(xf_y-y f_x)}\ dz.
\]
The   integral in the right-hand side is the sub-Riemannian area of $\Sigma$.

On the other hand, the constants $C_{R\epsilon\tau}$ and $D_{R\epsilon\tau}$
in \eqref{eq:k} are also asymptotic to $1/\epsilon$. Thus,
multiplied by $\epsilon$, inequalities \eqref{TP2} and \eqref{TP} pass to the
sub-Riemannian limit, see \cite{FLM}.
\end{remark}

The proof of Theorem \ref{thm:quant} is based on the foliation of the cylinder
$C_{R,\delta}$
by a family of CMC surfaces with quantitative estimates on the mean curvature.

\begin{theorem}
\label{thm:folsub}
For any $R>0$ and $0\leq\delta<R$, there exists a continuous  function
$u: C_{R,\delta}\to\R$
with level sets $S_\lambda=\big\{(z,t)\in C_{R,\delta} : u(z,t)=\lambda\big\}$,
$\lambda\in\R$,  such that the following claims hold:
\begin{itemize}
\item[(i)]   $u\in C^1(C_{R,\delta}\cap E_R)
\cap C^1(C_{R,\delta}\setminus E_R)$ and the normalized Riemannian  gradient
$\nabla u/|\nabla u|$ is
continuously defined on $C_{R,\delta}$.

\item[(ii)]   $\bigcup _{\lambda>R} S_\lambda =C_{R,\delta} \cap E_R$ and
$\bigcup
_{ \lambda\leq R} S_\lambda =C_{R,\delta} \setminus  E_R$.

\item[(iii)] Each $S_\lambda$ is a smooth  surface  with constant mean curvature
$H_{\lambda}=1/(\varepsilon \lambda)$ for 
$\lambda>R$ and $H_{\lambda} = 1/(\varepsilon R)$ for  $\lambda\leq R$.

\item[(iv)] For any point $(z,f(|z|;R)-t) \in S_\lambda$ with $\lambda>R$ we
have
\begin{equation}
 \label{H_R_2}
  1-\varepsilon R H_{\lambda}  (z,f(|z|;R)-t)\geq
\frac{t^2}{4Rk_{R\varepsilon\tau}^2+ f(0;R)^2} ,
\quad 
\textrm{when }\delta=0,
\end{equation}
and 
\begin{equation}
 \label{H_R_1}
  1-\varepsilon R H_{ \lambda} (z,f(|z|;R)-t)\geq
\frac{ \sqrt{\delta} t} {Rk_{R\varepsilon\tau}+ f(0; R)},
\quad
\textrm{when }0<\delta<R.
\end{equation}
\end{itemize}
\end{theorem}

\begin{proof}[Proof of Theorem \ref{thm:folsub}]
For points 
$(z,t) \in C_{R ,\delta}\setminus E_R$ we let
\begin{equation*}
\label{eq:uab}
u(z,t)=f(|z|; R)-t+R.
\end{equation*}
Then $u$   satisfies $u(z,t)\leq R$ for $ t \geq f(|z|;R)$
and $u(z,t)=R$ if $t=f(|z|;R)$. In order to define $u$ in the set
  $C_{R,\delta}\cap E_R$,
for $ 0\leq r<r_{R,\delta}$, $t_{R,\delta}<t<f(r; R)$,
and  $\lambda>R$ we consider  the function
\begin{equation}\label{fulla}
   F (r,t,\lambda)=f(r;\lambda)-f(r_{R,\delta};\lambda)+t_{R,\delta}-t .
\end{equation}
The function $F$ also depends on $\delta$.
We claim that for any point $(z,t)\in
C_{R,\delta}\cap E_R$ there exists a unique
$\lambda>R$ such that $F (|z|,t,\lambda)=0$.
In this case, we can define 
\begin{equation}
\label{eq:ubel}
u(z,t)=\lambda\quad\text{if and only if}\quad F (|z|,t,\lambda)=0.
\end{equation}
We prove the previous claim.
Let $(z,t)\in C_{R,\delta}\cap E_R$ and use the notation
$r=|z|$. First of all, we have  
\begin{equation}
\label{eq:cla}
\lim_{\lambda\to R^+}F (r,t,\lambda)=f(r;R)-t>0.
\end{equation}
We claim that we also have 
\begin{equation}
\label{eq:claim1}
\lim_{\lambda\to\infty}F (r,t,\lambda)=t_{R,\delta}-t<0.
\end{equation}
To prove this, we let
$f(r;\lambda)-f(r_{R,\delta};\lambda)=\frac{\varepsilon^2}{2\tau}[
f_1(\lambda)+f_2(\lambda)]$, where
\[\begin{split}
f_1(\lambda)&=\omega(\lambda)^2\Big[
\arctan(p(r;\lambda))-\arctan(p(r_{R,\delta};\lambda))\Big],\\
f_2(\lambda)&=\omega(r)^2\Big(p(r;\lambda)-p(r_{R,\delta};\lambda)\Big).
\end{split}
\]
Using the asymptotic approximation 
\[
\arctan(s)=\frac{\pi}{2}-\frac{1}{s}+\frac{1}{3s^3}+o\Big(\frac{1}{s^3}\Big),
\quad\text{as }s\to\infty,
\]
we obtain for $\lambda\to\infty$ 
\[
\begin{split}
f_1(\lambda)&=\lambda\varepsilon\tau(\omega(r_{R,\delta})-\omega(r)))+o(1),
\\
f_2(\lambda)&=\lambda\varepsilon\tau(\omega(r)-\omega(r_{R,\delta}))+o(1),
\end{split}
\]
and thus $f(r;\lambda)-f(r_{R,\delta};\lambda)=  o(1)$,
where $o(1)\to0$ as $\lambda\to\infty$.
Since $\lambda\mapsto F (r,t,\lambda)$ is continuous, \eqref{eq:cla} and
\eqref{eq:claim1} imply the existence of a solution $\lambda$ of
$F (r,t,\lambda)=0$. 
The uniqueness follows from  $\partial_\lambda
F(r,t,\lambda) <0$. 
This inequality can be proved starting from \eqref{f_R}  and we   skip the
details.
This finishes the proof of our initial claim.

Claims (i) and (ii) can be checked from the construction of $u$.
Claim (iii) follows, by Theorem \ref{3.1}, from the fact that $S_\lambda$ for
$\lambda>R$
is a vertical translation (this is an isometry of $H^1$) of the $t$-graph of
$z\mapsto f(z;\lambda)$.

We prove Claim (iv). For any  $(z,t)\in H^1$ such that  $r=|z|<r_{R,\delta}$
and 
$0\leq
t<f(r;R)-t_{R,\delta}$,  we
define
\begin{equation}
\label{pippo}
g_z(t)=u(z,f(r;R)-t)
=\lambda,
\end{equation}
where $\lambda\geq R$ is uniquely determined by the condition $(z,f(r;R)-t)\in
S_\lambda$.
Notice that $g_z(0)=u(z,f(r;R))=R$. 
We estimate the derivative of the  function $t\mapsto g_z(t)$.
From the identity $F(r,t,u(z,t)) = 0$, 
see \eqref{eq:ubel},  we compute $
\partial _t  u  (z,t) = (\partial_\lambda  F (r,t,u(z,t)) )^{-1}$ and
so, also using \eqref{fulla}, we find 
\begin{equation} \label{GG1}
  g'_z(t) = -  \partial_t u   (z, f(r;R)-t) =
\frac{-1}{\partial_\lambda F (r,f(r;R)-t,g_z(t))}.
\end{equation}
Now  from \eqref{palix} we compute  
\begin{equation} \label{GG2}
\begin{split}
 \partial _\lambda F (r,t,\lambda) & 
= -\e^ 3 \lambda
\int_r^{r_{R,\delta}} \frac{s\omega(s)}{(\lambda^2-s^2)^{3/2}}ds
\\
&
\geq -\e^ 3 \lambda \omega(r_{R,\delta}) \int_0^{r_{R,\delta}}
\frac{s }{(\lambda^2-s^2)^{3/2}}ds
\\
&
= -\e^ 3 
\omega(r_{R,\delta})\left[\frac{\lambda
}{\sqrt{\lambda^2 -  r_{R,\delta} ^2 } } -1\right ]
\\
&
 \geq -\e^3 \omega(R)
  \frac{\sqrt{R}}
 {\sqrt{\lambda - r_{R,\delta}}}.
\end{split}
\end{equation} 
In the last inequality, we used $r_{R,\delta} <R\leq \lambda$.
From \eqref{GG1}, \eqref{GG2} and with $k_{R\varepsilon\tau}$
as
in \eqref{eq:k}, we deduce that
\begin{equation}
\label{eq:disfz}
g_z'(t)\geq \frac{1} {k_{R\varepsilon\tau}} \sqrt{g_z(t)-r_{R,\delta}}.
\end{equation}
In the case  $\delta=0$,   \eqref{eq:disfz}  reads
$
g_z'(t)\geq \sqrt{g_z(t)-R} / k_{R\varepsilon\tau}$. Integrating this
differential inequality we obtain 
$
 g_z(t)\geq R+  {t^2}/ ({4 k_{R\varepsilon\tau}^2})$,
and thus  
\[
1-\varepsilon R H_{\lambda}(z,f(r;R)-t) =1-\frac{R}{g_z(t)} 
\geq\frac{t^2}{4Rk_{R\varepsilon\tau}^2+f(0;R)^2},
\]
that is Claim \eqref{H_R_2}.

If $0<\delta<R$, \eqref{eq:disfz} implies $g_z'(t)\geq
\sqrt{\delta}/k_{R\varepsilon\tau}$ and an integration gives  $g_{z}(t)\geq
\sqrt{\delta} \, t+ R/k_{R\varepsilon\tau}$.
Then we obtain
\[
1-\varepsilon R
H_{\lambda}(z,f(r;R)-t)=1-\frac{R}{g_z(t)} 
\geq\frac{\sqrt{\delta}}{Rk_{R\varepsilon\tau}+f(0;R)}t,
\]
that is Claim \eqref{H_R_1}.

\end{proof}

We can now prove Theorem \ref{thm:quant}, the last result of the paper.
The proof follows   the lines of \cite{FLM}.

 \begin{proof}[Proof of Theorem \ref{thm:quant}]
Let $u:C_{R,\delta}\to\R$, $0\leq\delta<1$,
be the function constructed in  Theorem \ref{thm:folsub}
and let  $S_\lambda=\{(z,t) \in C_{R,\delta} :u(z,t)=\lambda\}$, $\lambda\in\R$,
be
the leaves of
the foliation. Let  $\nabla u$ be  the Riemannian gradient of $u$.
The vector
field  
\[
V(z,t)=-\frac {\nabla u(z,t)}{|\nabla u(z,t)|},\quad (z,t) \in C_{R,\delta},
\] 
 satisfies the
following
properties:
\begin{itemize}
\item[i)]  $|V|= 1$.
\item[ii)] For $(z,t)\in\Sigma_R \cap C_{R,\delta}$ we have 
$V(z,t)= \nu_{ \Sigma _R}(z,t)$, where $\nu_{\Sigma_R} = \enn$ is the exterior
 unit normal to $\Sigma_R$.
\item[iii)] For any point $(z,t) 
\in S_\lambda $, $\lambda\in\R$,  the Riemannian divergence of $V$ satisfies
\begin{equation} \label{INEQ}
\begin{split}
 &\frac{1}{2}\mathrm{div} V(z,t) =H_{\lambda}(z,t) \leq
\frac{1}{\varepsilon R} \quad \text{for }\lambda>R,
\\
 &\frac{1}{2}\mathrm{div} V(z,t) =H_{\lambda}(z,t) =
\frac{1}{\varepsilon R}\quad\text{for }0< \lambda\leq R.
 \end{split}
\end{equation}
\end{itemize}

Let $\nu_\Sigma$ be the exterior unit normal to the surface $\Sigma=\partial E$.
By the Gauss-Green formula and \eqref{INEQ} it follows that 
\begin{align}
\label{eq:E-F}\nonumber
\mathcal L^3(E_R\setminus E)
& \geq \frac{\varepsilon R}{2}\int_{E_R\setminus E}\div V\;d\mathcal L^3 \\
\nonumber
&=\frac{\varepsilon R}{2}\Big(\int_{\Sigma_R \setminus \bar E}\la
V,\nu_{\Sigma_R}\ra\;d\A -\int_{\Sigma \cap E_R} \la
V,\nu_{\Sigma}\ra\;d\A   \Big)
\\
\nonumber
&\geq \frac{\varepsilon R}{2}\big( \A(\Sigma_R  \setminus \bar E  )-\A(\Sigma
\cap E_R)\big).
\end{align}
In the last inequality we used the Cauchy-Schwarz inequality and  the fact that
$\la V,\nu_{\Sigma _R}\ra=1$  
on $\Sigma_R\setminus \bar E$. 
By a similar computation we also have
\begin{align}
\nonumber
\mathcal L^3(E\setminus  E_R)& 
=\frac{\varepsilon R}{2}\int_{E\setminus  E_R}
\div V\;d\mathcal L^3
\\
\nonumber&
=\dfrac{\varepsilon R}{2}\left\{\int_{
\Sigma\setminus \bar E_R}\langle V,\nu_{ \Sigma }\rangle d\A 
-\int_{\Sigma_R  \cap E}\langle
V,\nu_{\Sigma_R}\rangle d\A \right\}
\\
\nonumber
&\leq  \frac{\varepsilon R}{2}\big( \A( \Sigma \setminus \bar E_R) -\A
(\Sigma_R\cap E)\big).
\end{align}
Using the inequalities above and the fact that $\mathcal L^3(E) =\mathcal
L^3(E_R)$, it  follows that:
\[
\begin{split}
\frac{\varepsilon R}{2}\big(\A (\Sigma_R  \setminus \bar E)-\A(\Sigma
\cap E_R)\big)
&\leq \frac{\varepsilon R}{2}\int_{E_R\setminus E}\div V\;d\mathcal L^3
\\
&=\mathcal L^3(E\setminus E_R)-\int_{E_R\setminus E}\Big(1-\frac{\varepsilon
R}{2}\div V\Big)\;d\mathcal L^3
\\
&\leq \frac{\varepsilon R}{2}\big(\A(\Sigma \setminus \bar E_R)
-\A(\Sigma_R\cap E) \big)  -\mathcal G(E_R\setminus E),
\end{split}\]
where we let 
\[
\mathcal G( E_R\setminus E)=
\int_{ E_R\setminus E}\Big( 1-\dfrac{\varepsilon R}{2}\div V\Big)\;d\mathcal
L^3.
\]
Hence, we obtain
\begin{equation} \label{F_E}
 \A(\Sigma) - \A(\Sigma_R)  \geq \frac{2}{\varepsilon R}\mathcal G(
E_R\setminus E).
\end{equation}

For any $z$ with $|z|<R-\delta$, we define
the vertical sections
 $ E_R^z =\{t\in\R : (z,t) \in  E_R\}$ and
$E^z =
\{t\in\R:(z,t) \in E\}$. 
By Fubini-Tonelli theorem, we
have
\[
 \begin{split}
\mathcal G( E_R\setminus E) &
= \int _{\{|z|<R\}} \int_{ E_R^z\setminus E^z} \Big( 1-\dfrac{\varepsilon
R}{2}\div V(z,t)\Big)dt  \,dz.
\end{split}
\]
The function  $t\mapsto \div V(z,t)$ is increasing, and thus 
letting $m(z) = \mathcal L^1 ( E_R^z\setminus E^z)$, by monotonicity we obtain
\[
\begin{split}
 \mathcal G( E_R\setminus E) & 
 \geq \int_{\{|z|<1\}}
\int_{f(|z|;R)-m(z)}^{f(|z|;R)}
\Big( 1-\dfrac{\varepsilon R}{2}\div V(z,t) \Big)dt \,dz
\\&
= \int_{\{|z|<1\}}
\int_{0}^{m(z)}
\left(1-\frac{R}{g_z(t)} \right) dt \,dz,
\end{split}
\]
where $g_z(t) = u(z, f(|z|;R)-t)$ is the function introduced in \eqref{pippo}.

When $\delta=0$, by the inequality \eqref{H_R_2}
and by H\"older inequality we find 
\begin{equation}\label{pix}
\begin{split}
\mathcal G( E_R\setminus E) & 
    \geq \frac{1}{4Rk_{R\varepsilon\tau}^2+f(0;R)^2} 
\int_{\{|z|<R\}}
    \int_{0}^{m(z)} t^2  dt \,dz
%
\\
&
\geq
\frac{1}{24\pi^2R^4(4Rk_{R\varepsilon\tau}^2+f(0;R)^2)
} \mathcal L^3 ( E\Delta E_R)
^3.
\end{split}
\end{equation}
From \eqref{pix}  and \eqref{F_E} we obtain \eqref{TP}.

By \eqref{H_R_1}, when $0<\delta<1$ the function $g_z$ satisfies the 
estimate $1-1/g_z(t) \geq
(\sqrt{\delta}/(k_{R\varepsilon\tau}+f(0;R)))t$ and  we
find
\begin{equation}\label{pox}
\begin{split}
\mathcal G( E_R\setminus E) & 
    \geq
\frac{\sqrt{\delta}}
{Rk_{R\varepsilon\tau}+f(0;R)}
    \int_{\{|z|<R\}}
    \int_{0}^{m(z)} t  \, dt \,dz 
\\
&
\geq  \frac{\sqrt{\delta}}{8\pi
R^2(Rk_{R\varepsilon\tau}+f(0;R))}
 \mathcal L^3 ( E\Delta E_R)
^2.
\end{split}
\end{equation}
From \eqref{pox} and  \eqref{F_E} we obtain  Claim \eqref{TP2}.

 \end{proof}

\end{document}